\documentclass[journal,twoside]{ieeecolor}
\usepackage{generic,enumerate}
\usepackage{cite}
\usepackage{amsmath,amssymb,amsfonts}
\allowdisplaybreaks
\usepackage{mathtools,comment}
\usepackage{dsfont}
\usepackage{graphicx}
\usepackage{algorithm}
\usepackage{algpseudocode} 
\usepackage{hyperref}
\usepackage{textcomp}
\let\originalleft\left
\let\originalright\right
\renewcommand{\left}{\mathopen{}\mathclose\bgroup\originalleft}
\renewcommand{\right}{\aftergroup\egroup\originalright}
\DeclareMathOperator*{\argmin}{arg\,min}
\newcommand{\bE}{\mathbb{E}}

\newcommand{\cS}{\mathcal{S}}

\newcommand{\bN}{\mathbb{N}}

\newcommand{\ust}{^{\star}}
\newcommand{\ub}{^{(\beta)}}
\newcommand{\bR}{\mathbb{R}}
\newcommand{\bZ}{\mathbb{Z}}
\newcommand{\bB}{\mathbb{B}}

\newcommand{\cI}{\mathcal{I}}

\newcommand{\bP}{\mathbb{P}}

\newcommand{\cT}{\mathcal{T}}
\newcommand{\cB}{\mathcal{B}}

\newcommand{\cR}{\mathcal{R}}

\newcommand{\cL}{\mathcal{L}}
\newcommand{\cF}{\mathcal{F}}
\newcommand{\cJ}{\mathcal{J}}
\newcommand{\cX}{\mathcal{X}}

\newcommand{\Tg}{\Tilde{\Gamma}}

\newcommand{\Tpi}{\Tilde{\pi}}
\newcommand{\Te}{\Tilde{e}}
\newcommand{\Tb}{\Tilde{b}}
\newcommand{\Tu}{\Tilde{u}}
\newcommand{\TV}{\Tilde{V}\ub}

\newcommand{\vp}{\varphi}
\newcommand{\Td}{\Tilde{d}}
\newcommand{\lf}{\left}
\newcommand{\rt}{\right}

\newcommand{\nal}[1]{\begin{align*}#1\end{align*}}
\newcommand{\al}[1]{\begin{align}#1\end{align}}
\newtheorem{assumption}{\textbf{Assumption}}

\newtheorem{definition}{Definition}[section]
\newtheorem{theorem}{Theorem}[section]
\newtheorem{proposition}{Proposition}[section]
\newtheorem{corollary}[proposition]{Corollary}
\newtheorem{lemma}[theorem]{Lemma}

\def\BibTeX{{\rm B\kern-.05em{\sc i\kern-.025em b}\kern-.08em
    T\kern-.1667em\lower.7ex\hbox{E}\kern-.125emX}}
\begin{document}
\title{Jointly Optimal Policies for Remote Estimation of Autoregressive Markov Processes over Time-Correlated Fading Channel}
\author{Manali Dutta, Rahul Singh, \IEEEmembership{Member, IEEE}, and Shalabh Bhatnagar, \IEEEmembership{Fellow, IEEE}
\thanks{Manali Dutta and Shalabh Bhatnagar are with the Department of Computer Science
and Automation, Indian Institute of Science, Bangalore 560012,
Karnataka, India. E-mail: manalidutta@iisc.ac.in, shalabh@iisc.ac.in. }
\thanks{E-mail: rahulsingh0188@gmail.com.}
}

\maketitle
\thispagestyle{empty}

\begin{abstract}
    We study a remote estimation setup with an autoregressive (AR) Markov process, a sensor, and a remote estimator. The sensor observes the process and sends encoded observations to the estimator as packets over an unreliable communication channel modeled as the Gilbert-Elliot (GE) channel. We assume that the sensor gets to observe the channel state by the ACK/NACK feedback mechanism only when it attempts a transmission while it does not observe the channel state when no transmission attempt is made. The objective is to design a transmission scheduling strategy for the sensor, and an estimation strategy for the estimator that are jointly optimal, i.e., they minimize the expected value of an infinite-horizon cumulative discounted cost defined as the sum of squared estimation error over time and the sensor’s transmission power. Since the sensor and the estimator have access to different information sets, this constitutes a decentralized stochastic control problem. We formulate this problem as a partially observed Markov decision process (POMDP) and show the existence of  jointly optimal transmission and estimation strategies that have a simple structure. More specifically, an optimal transmission strategy exhibits a threshold structure, i.e., the sensor attempts a transmission only when its belief about the channel being in a good state exceeds a threshold that depends on a certain error. 
    Moreover, an optimal estimation strategy follows a `Kalman-like' update rule. When the channel parameters are unknown, we exploit this structure to design an actor-critic reinforcement learning algorithm that converges to a locally optimal policy. Simulations show the learned policy performs close to a globally optimal one, with about a $5.5\%$ average relative gap across evaluated parameters.
\end{abstract}


\section{Introduction}
Wireless networked control systems (WNCS)~\cite{wang2023review} have gained significant traction in recent years due to their applicability in areas such as smart grids, autonomous vehicles, and industrial automation. The components in such a system communicate with each other via wireless communication channels. 
Such systems typically have multiple decision makers (DMs) that take decisions in order to achieve a shared objective~\cite{ge2017distributed}. Each decision maker uses only its own locally available information, collected either from the environment or from other components~\cite{ge2017distributed}, while making decisions. A fundamental challenge in optimizing the objective is that different DMs may have access to different sets of information. This gives rise to \emph{decentralized stochastic control}~\cite{nayyar2013decentralized}. The current work considers one such problem.

We consider a setup consisting of a sensor that observes an auto-regressive (AR) Markov process. The sensor encodes the observations into data packets before transmitting them to a remote estimator via an unreliable wireless communication channel modeled as a Gilbert-Elliot (GE) channel~\cite{gilbert1960capacity}. This formulation is commonly referred to as a remote state estimation problem~\cite{Lipsa2011remote}. In this setup there are two decision makers: the sensor and the estimator. We assume that the sensor does not continually probe the channel, and hence gets to know the channel state only when an acknowledgment signal is sent by the estimator after a transmission attempt. The sensor has to decide when to transmit an observation to the estimator based on the information currently available with it. Since the sensor may not continually transmit packets, the estimator needs to estimate the current state of the AR process based only on its  locally available information. The objective is to design transmission and estimation strategies \footnote{We use the words policy and strategy interchangeably.} that are jointly optimal, i.e., they  minimize the expected value of
an infinite-horizon cumulative discounted cos~\cite{Hernandez2012discrete} that is the sum of the squared estimation error and the transmission power consumed by the sensor. The main challenge encountered while solving this problem is that the sensor and the estimator do not have the same information, and this gives rise to a decentralized decision-making problem.

The above problem was partially studied in~\cite{dutta2023optimal}. The remote estimator in~\cite{dutta2023optimal} is held fixed, and a simplified problem is studied -- one in which only the sensor has to design an optimal transmission policy. The sensor maintains a certain belief state which is the conditional probability that the channel state is good. A structural result is shown which says that there exists an optimal transmission policy that is of the threshold-type with respect to the belief state, and hence the sensor transmits only when the current belief exceeds a certain threshold. The formulation in~\cite{dutta2023optimal} involves only a single decision maker (the sensor) and therefore does not encounter the additional complexity addressed in the present work. Also, the resulting policy there may not be jointly optimal with the estimator. 

Remote state estimation of a Markov process over a communication channel which involves a single decision maker has been studied earlier in~\cite{xu2004optimal, shi2012optimal, wu2012event, han2015stochastic}. 
In~\cite{xu2004optimal, shi2012optimal}, the sensor acts as a decision maker while the estimator is fixed and taken to be ``Kalman-like''. However~\cite{xu2004optimal} assumes that the communication channel is ideal, so that each transmission is successful, while~\cite{shi2012optimal} consider a packet-dropping channel where each packet sent by the sensor is successfully received with a probability that depends on the transmission power level chosen.~\cite{xu2004optimal} shows the existence of an optimal transmission policy that has a threshold structure, and hence transmits only when the magnitude of current estimation error exceeds a certain threshold. In ~\cite{shi2012optimal}, the sensor decides whether to transmit using low or high transmission power under a limited high-power usage constraint. It then shows that, under such a constraint, there exists an optimal policy that allocates all high-power transmissions to the last available slots within the finite horizon. In \cite{wu2012event} and \cite{ han2015stochastic}, the authors fix the transmission policy of the sensor while optimizing the estimator. They show that the optimal estimator turns out to be a modified Kalman estimator~\cite{kalman1960new}. 

The design of jointly optimal transmission and estimation policies has been explored in \cite{Lipsa2011remote},~\cite{Nayyar2013optimal, Ren2017infinite, Chakravorty2019remote}. In \cite{Lipsa2011remote}, remote estimation of a Markov process with Gaussian noise is considered. It uses tools from majorization theory~\cite{hajek2008paging} to show that there exists a threshold-type transmission policy that transmits only when the error magnitude exceeds a prescribed threshold, and is jointly optimal with a Kalman-like estimator. However, it assumes that the communication channel is ideal and optimizes over a finite horizon. \cite{Nayyar2013optimal} considers an energy harvesting sensor with a limited battery capacity that observes a finite state Markov process. The sensor communicates to the estimator via an ideal communication channel. Using arguments from~\cite{Lipsa2009optimal} and~\cite{hajek2008paging}, it shows the existence of a threshold-type transmission strategy with respect to the magnitude of the Markov process that is jointly optimal with a Kalman-like estimator over a finite horizon. The value of the threshold at each time depends on the sensor’s current energy level and the conditional probability of the energy level and Markov state conditioned on the estimator’s current available information. It also extends the result to a multidimensional Gaussian source.~\cite{Ren2017infinite} also shows a similar structural result where an optimal transmission policy has a threshold structure with respect to an innovation error that quantifies how much the actual Markov process state has deviated from the estimator's predicted state. In contrast to~\cite{Lipsa2011remote} and~\cite{Nayyar2013optimal}, it considers a packet dropping channel that is modeled as a time-homogeneous Markov chain and considers minimizing an infinite horizon average cost. The value of the threshold depends on the current channel state and the conditional probability density function (pdf) of the Markov process conditioned on the information available with the estimator. However, it assumes that the sensor knows the current channel state at each time while making a decision. Even with these structural results, the policy search space still remains difficult. This is because in~\cite{Nayyar2013optimal}, the decision taken by the sensor at each time depends on its entire observation history which requires high storage capacity, and a threshold transmission policy of~\cite{Ren2017infinite} depends on the conditional pdf that may require high computation power. \cite{Chakravorty2019remote} addresses this issue by further reducing the policy space. It also proves the existence of a threshold-type transmission policy (with respect to error), where the threshold values depend only on the current channel state, together with a Kalman-like estimator with which it is jointly optimal. But it assumes that the sensor knows the channel state perfectly with a delay of one unit. In practice, continually probing a channel for accessing its current state may lead to overhead expenses and consume resources~\cite{chang2007optimal}. Hence, such an assumption might not be feasible in resource constrained networks. We study a more realistic setup than~\cite{Chakravorty2019remote}, where we assume that the sensor does not observe the channel state directly. Instead, only in the case when a transmission attempt is made by the sensor, it gets to know the channel state from an acknowledgment signal sent by the transmitter. Thus, when there is no transmission attempt, the sensor does not know the channel state.

This work extends the results of our earlier conference paper~\cite{dutta2023optimal} in the following directions:

\begin{enumerate}
    \item We generalize the additive noise process in the AR source observed by the sensor. Specifically, we allow i.i.d. noise with a symmetric and unimodal distribution. This is a generalization of the  assumption in \cite{dutta2023optimal} where a Gaussian noise process was assumed.
    \item In contrast to~\cite{dutta2023optimal} where the optimization is only over the class of transmission policies while keeping the estimator fixed, in this work, we jointly optimize over the sensor transmission policy and the estimator for the average-cost problem. The structural result for the transmission policy shows that transmission attempt is made only when the belief state (the conditional probability that the channel state is good conditioned on the information available with the sensor) exceeds a threshold. This threshold is a function of the error as defined in~\eqref{def:e}. Moreover the estimator is shown to be Kalman-like. Note that~\cite{dutta2023optimal} derives optimal transmission policy for the transmitter under the assumption that the estimator is Kalman-like. The policy thus obtained, and the Kalman filter may not be jointly optimal. However, as is shown in this work, the solution of~\cite{dutta2023optimal} despite being for the discounted cost setting is indeed optimal.   
    \item When the system parameters such as the Markovian channel transition probabilities are unknown, we use the structural result to design an actor-critic (AC) learning algorithm for minimizing the cost. Specifically, the threshold curve of a policy is parameterized by finitely many parameters, and these are tuned using AC. Empirical results show that the performance of the AC algorithm is comparable to the optimal performance yielded by the relative value iteration (RVI) algorithm~\cite{Hernandez2012discrete}. However, the RVI assumes perfect knowledge of the channel parameters.
\end{enumerate}

\noindent
\emph{Notations:} We use $\bN, \bZ_+$, $\bR$, $\bR_+$ to denote the set of natural numbers, non-negative integers, real numbers, and non-negative real numbers, respectively. $\bP(\cdot)$ denotes the probability of an event. For a sigma-algebra $\cF$, $\bP[\cdot | \cF]$ and $\bE[\cdot | \cF]$ represent the conditional probability and conditional expectation conditioned on $\cF$.~For an event $A$, $\mathds{1}(A)$ denotes its indicator random variable .~$\delta_x(\cdot)$ is the delta function with unit mass at $x$.~$||v||_{\infty}$ denotes the infinity norm of the vector $v$. Further,  $\ast$ is the convolution operator.

\section{Problem Formulation}\label{sec:prob_form}
\begin{figure}[t]
	\begin{centering}
		\includegraphics[scale=0.5]{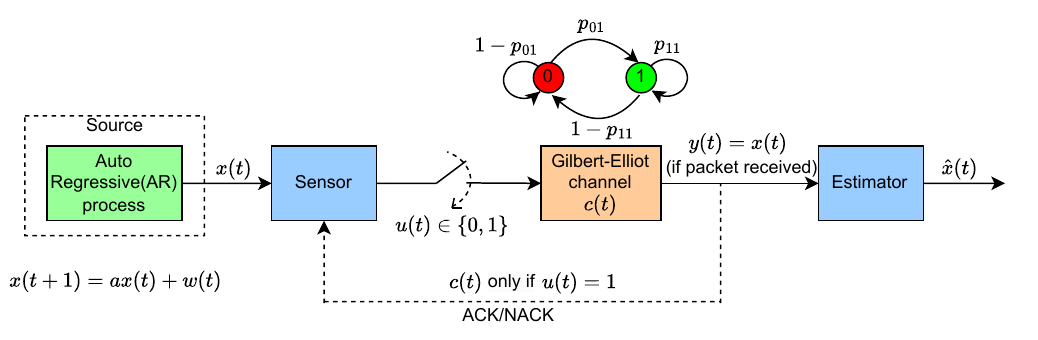} 
		\par\end{centering}
	\caption{Remote state estimation setup.}
	\label{fig:setup}
\end{figure}
We consider a networked system comprising of an AR process, a sensor and a remote estimator as shown in Fig.~\ref{fig:setup}. The sensor observes the AR process $\{x(t)\}_{t \in \bZ_+}$. The state of the process $x(t)\in \bR$, $t\geq 0$, evolves as follows with $x(0) \in \bR$ as the initial state:  
\al{
x(t+1) = a x(t) + w(t),~t \geq 0, \label{eq:source}
}
where $a \in\bR$ and the noise process $\{w(t)\}_{t \in \bZ_{\geq 0}}$ satisfies the following assumption:

\begin{assumption} \label{assump:noise}
    $\{w(t)\}_{t \in \bZ_+}$ is distributed according to a symmetric (about the origin) and unimodal distribution with pdf denoted by $\varphi$ and a finite second moment, i.e., $\sigma_w^2 := \bE[w(t)^2] < \infty$.
\end{assumption}

The sensor encodes its observations into data packets and transmits these to a remote estimator via an unreliable wireless channel. This wireless channel is modeled as a GE channel~\cite{gilbert1960capacity}. We denote the state of the channel at time $t$ by $c(t)\in \{0,1\}$. Here, $c(t)=0$ denotes that the channel is in a `bad state' and any transmission attempt is unsuccessful, while $c(t)=1$ denotes that the channel is in a `good state' and any packet transmitted at $t$ will be delivered to the estimator. We assume that $\{c(t)\}_{t\in\bZ_+}$ is a Markov process with parameters $p_{01}, p_{11} \in (0,1)$, where $p_{01}$ denotes probability that the channel is in state 1 at the next time step given that it is currently in state $0$, and $p_{11}$ denotes that the channel is in state $1$ at the next time step given that it is currently in state $1$. Let $u(t)\in \{0,1\}$ denote the action chosen by the sensor, where $u(t)=1$ corresponds to the sensor's decision to attempt a packet transmission at time $t$ while $u(t)=0$ indicates no transmission attempt by the sensor at time $t$. We assume that each transmission attempt consumes $\lambda$ units of power\slash resource.~Let $y(t) \in \bR \cup \Xi$ denote the observation made by the estimator at time $t$, where the event $y(t) = \Xi$ denotes that the estimator did not receive any data packet. This may occur either because no transmission attempt was made by the sensor, or an attempt was made but was not successful since the channel state was bad.~Let $\hat{x}(t)$ denote the estimate of $x(t)$ at the remote estimator.~The sensor does not observe the channel state $c(t)$.~If there is a successful transmission at time $t$, then the estimator sends an acknowledgment to the sensor. Hence, if $u(t)=1$, then the channel state $c(t)$ at $t$ is known to the sensor at time $t+1$. However, if $u(t)=0$, then no such acknowledgment is sent, so no channel state information is delivered to the transmitter. We let $z(t)$ denote the delayed channel state information delivered to the sensor upon a transmission attempt made at $t-1$. The information delay here is of one time unit. 
\subsection{Decision makers}\label{sec:dec_make}

In this section, we formulate the optimization problem of interest. We begin by describing the information set available with the two decision makers, i.e., the sensor and the estimator. 

\emph{Information available at the sensor}: At time $t$, the sensor has access to $\{x(s)\}_{s=0}^t \cup \{u(s), z(s), y(s)\}_{s=0}^{t-1}$. Let $\cF(t):= \sigma(\{x(s), u(s), z(s), y(s)\}_{s=0}^{t-1}, x(t)\})$, $t \in \bN$. We define the channel belief state $b(t) : =\bE[c(t) \mid \cF(t)]$. Thus, the information available at the sensor at time $t$ is,
\al{
\cI^s(t)=\lf(\{x(s),u(s),z(s),y(s),b(s)\}_{s=0}^{t-1}, x(t),b(t)\rt).\label{eq:info_sched}
}
The sensor at time $t$ makes decision $u(t)$ according to the information available to it till time $t$, i.e.,
\al{
u(t) = f_t(\cI^s(t)),\label{eq:u}
}
where $f_t : \cI^s(t) \rightarrow u(t)$, $u(t)\in \{0,1\}$ is a measurable function. The collection $f:= (f_0,f_1,\ldots)$ is a scheduling policy that at each time $t$ picks the action of whether or not to transmit a packet depending on the information $\cI^s(t)$ available.

\emph{Information available at the estimator}: At each time $t$ the estimator has access to the following information,
\al{
\cI^e(t)= \lf(\{y(s)\}_{s=0}^{t}\rt).
}
Note that although $y(t)$ is known to the estimator at time $t$, it does not know $z(t)$. This is because to compute $z(t)$, the estimator requires the knowledge of $u(t)$ since $z(t)= c(t)$ if $u(t)=1$ and $y(t) \neq \Xi$, else $z(t)=\Tilde{\Xi}$, where $\Tilde{\Xi}$ implies that no channel information is known. The sensor can, however, estimate $z(t)$ at time $t+1$ based on the feedback $y(t)$ from the estimator and $u(t)$.

The estimator produces its state estimate according to the information available to it at time $t$, i.e.,
\al{
\hat{x}(t)= g_t(\cI^e(t)),\label{eq:hatx}
}
where $g_t: \cI^e(t) \rightarrow \hat{x}(t) \in \bR$ is also a measurable function. The collection $g:= (g_0,g_1,\ldots)$ constitutes an estimation strategy for the estimator.

\subsection{Optimization Problem}
We consider the following optimization problem.



\emph{Problem 1:} For the model described in Section~\ref{sec:prob_form}, we consider here the problem of designing a scheduling policy $f$ for the sensor and an estimation strategy $g$ in order to solve the following \emph{discounted cost} problem:
\al{
	\min_{(f,g)}~\bE_{(f,g)} \left(\sum_{t=0}^{\infty} \beta^{t}\left((x(t)-\hat{x}(t))^2+\lambda u(t) \right)\right),\label{def:obj}
}
where$\bE_{(f,g)}$ denotes that the expectation is taken with respect to the measure induced by the policies $(f,g)$ and $\beta \in (0,1)$ is the discount factor. Also, recall that $\lambda>0$ denotes the number of units of power consumed in a transmission attempt.

\section{Preliminary Results}

We first show that we can restrict our attention to scheduling policies of the form discussed in the following lemma.
\begin{lemma}
    For the purpose of solving~\eqref{def:obj} and for any estimation strategy of the form~\eqref{eq:hatx}, there is no loss of optimality in restricting the class of scheduling policies~\eqref{eq:u} to the following form:
    \al{
    u(t)=f_t(x(t),b(t),\{y(s)\}_{s=0}^{t-1}).\label{eq:u_restrict}
    }
\end{lemma}
\begin{proof}
    We will prove that $\{(x(t),b(t),\{y(s)\}_{s=0}^{t-1})\}$ is a Markov decision process (MDP) controlled by $u(t),t\geq 0$. For this, it is sufficient to show that $\Tilde{\cI}^s(t):=(x(t),b(t),\{y(s)\}_{s=0}^{t-1})$ is an information state~\cite{Kumar2015stochastic} at the sensor, i.e., the following two conditions are satisfied:

    \begin{enumerate}
        \item[(i)] $\bP(\Tilde{\cI}^s(t+1) \mid \cI^s(t),u(t)) = \bP(\Tilde{\cI}^s(t+1) \mid \Tilde{\cI}^s(t),u(t))$, 

        \item[(ii)] $\bE[(x(t)-\hat{x}(t))^2 + \lambda u(t) \mid \cI^s(t), u(t)] \\
       \mbox{ }  = \bE[(x(t)-\hat{x}(t))^2 + \lambda u(t) \mid \Tilde{\cI}^s(t), u(t)]$.
    \end{enumerate}
    
    We consider (i) first.
    \al{
    & \bP(\Tilde{\cI}^s(t+1) \mid \cI^s(t),u(t)) \notag\\
    & = \bP(x(t+1) \mid x(t)) \bP(b(t+1) \mid b(t),y(t),u(t)) \notag\\
    & \times \bP(y(t) \mid x(t),b(t),u(t)) \mathds{1}({\{\{\Tilde{y}(s)\}_{s=0}^{t-1} = \{\Check{y}(s)\}_{s=0}^{t-1}\}})\notag\\
    & = \bP(\Tilde{\cI}^s(t+1) \mid \Tilde{\cI}^s(t),u(t)), \notag
    }
    where $\{\Tilde{y}(s)\}_{s=0}^{t-1}$ and $\{\Check{y}(s)\}_{s=0}^{t-1}$ are any given realizations of the random variables $\{{y}(s)\}_{s=0}^{t-1}$ corresponding to $\Tilde{\cI}^s(t+1)$ and $\cI^s(t)$, respectively. The first equality follows from the observation that 
    \al{
    &\bP(y(t) = y\mid x(t),b(t),u(t))\notag\\
    &=
    \begin{cases}
    b(t) &\mbox{if } u(t)=1,y \neq \Xi,\\
    (1-b(t)) &\mbox{if } u(t)=1,y = \Xi,\\
    1 &\mbox{if } u(t)=0,y = \Xi,\\
    0 &\mbox{otherwise}.\label{eq:y}
\end{cases}
    }
    Thus (i) is true. 

    Next, we consider (ii). From~\eqref{eq:hatx}, we have that the conditional expectation on the left hand side of (ii) depends upon $\bP(\cI^e(t) \mid \cI^s(t),u(t))$. Thus, (ii) follows  from~\eqref{eq:y}. This completes the proof of the lemma.
\end{proof}

For our convenience, we define a process $\{v(t)\}_{t \in \bZ_+}$ as follows:
\al{
& v(-1) =x_0, \label{eq:v(-1)} \\
& v(t)=
\begin{cases}
    av(t-1) &\mbox{ if } y(t) = \Xi,\\
    y(t) &\mbox{ otherwise} \label{eq:v}
\end{cases} \quad \text{for } t \geq 0.
}
Next, we define three more processes $\{e(t)\}_{t \in \bZ_+}$, $\{{e}^+(t)\}_{t \in \bZ_+}$, and $\{\hat{e}(t)\}_{t \in \bZ_+}$ as follows:
\al{
&e(t):=x(t)-av(t-1), {e}^+(t):=x(t)-v(t), \label{def:e}\\
& \hat{e}(t):=\hat{x}(t)-v(t). \notag
}
Note that $\{{e}^+(t)\}_{t \in \bZ_+}$ and $\{e(t)\}_{t \in \bZ_+}$ are related as 
\al{
& e^+(t)=
\begin{cases}
    e(t) &\mbox{ if } y(t) = \Xi,\\
    0 &\mbox{ otherwise,} \label{eq:e^+_t} 
\end{cases} \\
& e(t+1)=ae^+(t) + w(t). \label{eq:e^+_t,e_t}
}
Now, using~\eqref{eq:source} and~\eqref{eq:v(-1)}-\eqref{eq:v}, the evolution of the process $\{e(t)\}_{t \in \bZ_+}$ can be written as follows: 
\al{ 
& e(0) =0 \notag\\
	&e(t+1) =
	\begin{cases}
		ae(t)+w(t) & \text{if } y(t)=\Xi, \\
		w(t) & \text{otherwise,}\label{def:error_evolve}
	\end{cases} \quad \text{for } t \geq 0.
}

Since $x(t)-\hat{x}(t)={e}^+(t)-\hat{e}(t)$, we can equivalently write the single-stage cost as $(x(t)-\hat{x}(t))^2 + \lambda u(t)= ({e}^+(t)-\hat{e}(t))^2 + \lambda u(t)$.

\begin{definition}
    Policies $(f,g)$ and $(\Tilde{f}, \Tilde{g})$ are said to be \emph{equivalent} for the purpose of solving~\eqref{def:obj} if 
    \nal{
    & \bE_{(f,g)} \left(\sum_{t=0}^{\infty} \beta^{t}\left((x(t)-\hat{x}(t))^2+\lambda u(t) \right)\right) \\
    & = \bE_{(\Tilde{f},\Tilde{g})} \left(\sum_{t=0}^{\infty} \beta^{t}\left((x(t)-\hat{x}(t))^2+\lambda u(t) \right)\right).
    }
\end{definition}

In the following lemma, we show that the processes introduced above allow us to construct equivalent transmission and estimation strategies.
\begin{lemma} \label{lemma:tilde_fg}
    Consider Problem 1. Then, for scheduling and estimation strategies of the form~\eqref{eq:u_restrict} and~\eqref{eq:hatx}, respectively, there exist equivalent scheduling and estimation strategies of the following form:
    \al{
    u(t) &=\Tilde{f}_t(e(t),b(t),\{y(s)\}_{s=0}^{t-1}), \label{eq:u(e,b,y)} \\
    \hat{x}(t) &= \Tilde{g}_t(\{y(s)\}_{s=0}^{t}) + v(t). \label{eq:hatx_v}
    }
    
    Conversely, if the scheduling and estimation strategies have the form~\eqref{eq:u(e,b,y)} and~\eqref{eq:hatx_v}, respectively, then there exist equivalent scheduling and estimation strategies of the form~\eqref{eq:u_restrict} and~\eqref{eq:hatx}, respectively.
\end{lemma}
\begin{proof}
    First, we note from~\eqref{eq:v} that $v(t)$ is measurable with respect to $\sigma(\{y(s)\}_{s=0}^t)$. As a result, at time $t$ the sensor can compute $e(t)$ because it has access to $x(t)$ and $\{y(s)\}_{s=0}^{t-1}$ (and hence $v(t-1)$). 
    
    We use $\Tilde{f},\Tilde{g}$ to denote $\Tilde{f}:=(\Tilde{f}_0,\Tilde{f}_1,\ldots)$ and $\Tilde{g}:=(\Tilde{g}_0,\Tilde{g}_1,\ldots)$. Now, we consider the following two cases.

    1) Given $f,g$ of the form~\eqref{eq:u_restrict},~\eqref{eq:hatx}, respectively, we construct
    \al{
    & \Tilde{f}_t(e(t),b(t),\{y(s)\}_{s=0}^{t-1}) \notag \\
    & = f_t(e(t)+a h_{t-1}(\{y(s)\}_{s=0}^{t-1},b(t),\{y(s)\}_{s=0}^{t-1}) \notag \\
    & \text{and} \notag \\
    & \Tilde{g}_t(\{y(s)\}_{s=0}^{t}) =g_t(\{y(s)\}_{s=0}^{t})-h_t(\{y(s)\}_{s=0}^{t}).\notag
    }
    Then, by construction $\Tilde{f}, \Tilde{g}$ are equivalent to $f,g$, respectively.

    2) Given $\Tilde{f},\Tilde{g}$ of the form~\eqref{eq:u(e,b,y)},~\eqref{eq:hatx_v}, respectively, we define
    \al{
    & f_t(x(t),b(t),\{y(s)\}_{s=0}^{t-1}) \notag\\
    & = \Tilde{f}_t(x(t)-a h_t(\{y(s)\}_{s=0}^{t-1}), b(t), \{y(s)\}_{s=0}^{t-1}) \notag \\
    & \text{and} \notag \\
    & g_t(\{y(s)\}_{s=0}^{t}) =\Tilde{g}_t(\{y(s)\}_{s=0}^{t}) + h_t(\{y(s)\}_{s=0}^{t}).\notag
    }
    Then, by construction $f,g$ are equivalent to $\Tilde{f}, \Tilde{g}$, respectively. This concludes the proof.
\end{proof}
An immediate consequence of the above lemma is that we can now assume that the sensor observes the process $\{e(t)\}_{t \in \bZ_+}$ instead of $\{x(t)\}_{t \in \bZ_+}$, and the estimator generates the following estimate at time $t$,
\al{
\hat{e}(t)= \hat{x}(t)-v(t) = \Tilde{g}_t(\{y(s)\}_{s=0}^{t}). \label{eq:est_e}
}
Consequently, we will restrict ourselves to scheduling strategies of the form~\eqref{eq:u(e,b,y)}, and estimation strategies of the form~\eqref{eq:est_e}. Next, we note that at the beginning of time $t$, the sensor knows $(e(t),b(t),\{y(s)\}_{s=0}^{t-1})$, while the estimator knows only $\{y(s)\}_{s=0}^{t-1}$. This type of information structure has been referred to as  \emph{partial history sharing}~\cite{nayyar2013decentralized}. One approach to solving such a decentralized stochastic control problem is to view the problem from the perspective of a coordinator (or decision maker) who knows the \emph{common information}~\cite{nayyar2013decentralized} among all the other decision makers. We will use this approach as is discussed in the next section. In our case, the estimator acts as a coordinator since it knows the common information $\{y(s)\}_{s=0}^{t-1}$ between the sensor and the estimator.

\subsection{Equivalent Problem}
We will now formulate an equivalent optimization problem from the estimator's perspective for the model described in Section~\ref{sec:prob_form}. At the end of time $t-1$, the estimator chooses an estimate based on the information $\{y(s)\}_{s=0}^{t-1}$ as follows:
\nal{
\hat{e}(t-1) = \Tilde{g}_t(\{y(s)\}_{s=0}^{t-1}).
}
Next, at the beginning of time $t$, the estimator chooses mapping $\Gamma_t: \bR \times [0,1] \rightarrow u_t \in \{0,1\}$ utilizing the information $\{y(s)\}_{s=0}^{t-1}$. This mapping is described by 
\al{
\Gamma_t = \ell_t(\{y(s)\}_{s=0}^{t-1}). \label{eq:l}
}
Then, at time $t$, the sensor makes the decision $u(t)=\Gamma_t(e(t),b(t))$. We refer to $\Gamma_t$ as the \emph{estimator's prescription} to the sensor. In addition, we call $\ell:=\{\ell_0,\ell_1,\ldots\}$ as the coordination strategy. We now formulate the following optimization problem for the estimator.

\emph{Problem 2:} For the model given in Section~\ref{sec:prob_form}, the goal is to find coordination strategy $\ell$, and an estimation strategy $\Tilde{g}$ so as to solve
\al{
	\min_{(\ell,\Tilde{g})}~\bE_{(\ell,\Tilde{g})} \left(\sum_{t=0}^{\infty} \beta^{t}\left(({e}^+(t)-\hat{e}(t))^2+\lambda u(t) \right)\right).\label{def:new_obj}
}
The following lemma shows the equivalence of Problem 1 with Problem 2.
\begin{lemma}
    Given scheduling and estimation strategies $(\Tilde{f}, \Tilde{g})$ corresponding to Problem 1, there exist coordination and estimation strategies $(\ell,\Tilde{g})$ for Problem 2 that achieve the same expected cost as that incurred by $(\Tilde{f}, \Tilde{g})$ in Problem 1. 

    Conversely, for any $(\ell,\Tilde{g})$ in Problem 2, there exist $(\Tilde{f}, \Tilde{g})$ for Problem 1 that achieve the same expected cost as that incurred by $(\ell,\Tilde{g})$ in Problem 2.
\end{lemma}

\begin{proof}
    Given $\Tilde{f}, \Tilde{g}$ of the form~\eqref{eq:u(e,b,y)},~\eqref{eq:hatx_v}, respectively, define the coordination strategy for Problem 1 as follows,
    \nal{
    \ell_t(\{y(s)\}_{s=0}^{t-1})=\Tilde{f}(\cdot,\cdot,\{y(s)\}_{s=0}^{t-1}).
    }
    Then, by construction, $\ell,\Tilde{g}$ achieve the same expected value in Problem 2 as $\Tilde{f}, \Tilde{g}$ in Problem 2.

    Now, for the converse part, consider any $\ell,\Tilde{g}$ of the form~\eqref{eq:l}~\eqref{eq:hatx_v}, respectively, corresponding to Problem 1. Construct a scheduling strategy by letting 
    \nal{
    \Tilde{f}(\cdot,\cdot,\{y(s)\}_{s=0}^{t-1})=\ell_t(\{y(s)\}_{s=0}^{t-1}).
    }
    The proof then follows from construction.
\end{proof}
As a result of the above lemma, we will now focus on solving Problem 2. We show here that the estimator's optimization problem (Problem 2) can be formulated as a POMDP~\cite{Krishnamurthy2016partially}.

\textbf{POMDP formulation of Problem 2}: We begin by describing the components of this POMDP.

1) \emph{States}: The pre-scheduling state at time $t$ is $(e(t),b(t))$. The post-scheduling state is $(e^+(t),b(t))$.

2) \emph{Actions}: $\Gamma_t$ is the pre-scheduling action and $\hat{e}(t)$ is the post-scheduling action at time $t$.

3) \emph{Evolution of states}: The evolution of the state $(e^+(t),b(t))$ depends on $(e(t),b(t))$ and the pre-scheduling action $\Gamma_t$ as shown in~\eqref{eq:e^+_t}.
~The post-scheduling state then evolves from $(e^+(t),b(t))$ to $(e(t+1),b(t+1))$, see \eqref{eq:e^+_t,e_t}, where the evolution of $b(t+1)$ is as below.
\al{ \label{def:beliefevolv}
	b(t+1)  = 
	\begin{cases}
		p_{11} & \text{if } u(t)=1, c(t) = 1, \\
		p_{01} & \text{if } u(t)=1, c(t) = 0, \\
        \cT(b(t)) & \text{if } u(t)=0,
	\end{cases}
 }

4) \emph{Observation}: The observation $y(t)$ depends on $(e(t),b(t))$ and $\Gamma_t$ as follows:
\al{
&\bP(y(t) =y \mid e(t),b(t),\Gamma_t) \notag\\
&=
\begin{cases}
    b(t) &\mbox{ if } \Gamma_t(e(t),b(t))=1, y \neq \Xi,\\
    (1-b(t)) &\mbox{ if } \Gamma_t(e(t),b(t))=1, y = \Xi,\\
    1 &\mbox{ if } \Gamma_t(e(t),b(t))=0, y = \Xi,\\
    0 &\mbox{ otherwise.} \label{eq:prob_y}
\end{cases}
}

5) \emph{Single-Stage or Instantaneous cost}: The transmission power consumed is $\lambda$ if $\Gamma_t(e(t),b(t))=1$, otherwise it is 0. The squared estimation error is $({e}^+(t)-\hat{e}(t))^2$ and is a function of the post-scheduling state and post-scheduling action.

\textbf{Belief space} We now discuss the space in which $b(t)$ resides. We fix an initial state $b(0) \in \cB$. If no transmissions are attempted on the interval $[0, t_0 -1]$, i.e., $u(s) = 0$ for $0 \leq s \leq t_0 - 1$, $t_0 \in \bN$, then it follows from~\eqref{def:beliefevolv} that 
\nal{b(t_0) = \cT^{t_0}(b(0)),
}
where $\cT^{t_0}(\cdot)$ denoted the $t_0$-fold composition of $\cT$ and $T^0(b):= b$. Moreover, if a transmission is attempted at time $t_0$, i.e., $u(t_0) = 1$, then $b(t_0 +1) \in \{p_{01}, p_{11}\}$. Hence for $t > t_0$, $b(t) \in \{\cT^{s}(p_{01})\}_{s \in \bZ_+} \cup \{\cT^{s}(p_{11})\}_{s \in \bZ_+}$. Combining these possibilities, the belief space denoted by $\cB$ can be written as follows $b(t) \in \cB:= \{\cT^{s}(b(0))\}_{s \in \bZ_+} \cup \{\cT^{s}(p_{01})\}_{s \in \bZ_+} \cup \{\cT^{s}(p_{11})\}_{s \in \bZ_+}$. Also, we note that $\cB \subseteq [0,1]$.

As a result of the above POMDP formulation, we can define the estimator's information states~\cite{Kumar2015stochastic} as follows.
\begin{definition} \label{def:pre_post_beliefs}
\begin{enumerate}[(i)]
    \item The ``pre-scheduling belief'' at time $t$, namely $b^{(pr.)}_t(\cdot,\cdot): \bR \times [0,1] \rightarrow \bR_+$ is the conditional probability density of $(e(t),b(t)) \in \bR \times [0,1]$ conditioned on $\{y(s),\Gamma_s\}_{s=0}^{t-1}$.
    \item The ``post-scheduling belief'' at time $t$, $b^{(po.)}_t(\cdot,\cdot): \bR \times [0,1] \rightarrow \bR_+$, is the conditional probability density of $(e^+(t),b(t)) \in \bR \times [0,1]$ conditioned on $\{y(s),\Gamma_s\}_{s=0}^{t}$.
\end{enumerate}
\end{definition}
These beliefs can be recursively updated as discussed next.
\begin{lemma}\label{lemma:prepost_evolve}
    The beliefs $b^{(pr.)}_t$ and $b^{(po.)}_t$ evolve as described here below.
    \al{
    1)~ &b^{(pr.)}_{t+1}(e,b)=(1/|a|) \int_{e' \in \bR} \sum_{b' \in \cB} \bP(b_{t+1}=b \mid b_t =b') \notag \\
    & \qquad \qquad \times \varphi(e-e') b^{(po.)}_t(e'/a, b') \, de' \notag \\
    & = \sum_{b' \in \cB} \bP(b_{t+1}=b \mid b_t =b') \eta_t(e,b') \ast \varphi \notag,}
    where $a$ is the AR process parameter~\eqref{eq:source}, $\eta_t(e,b') = (1/|a|) b^{(po.)}_t(e/a,b')$ is conditional probability density of $(ae^+(t),b(t))$ (whereas $b^{(po.)}_t(e,b)$ is the conditional probability density of $(e^+(t),b(t))$), $\ast$ is the convolution operator, and $\varphi$ is the pdf of $w(t)$.~We denote this transformation as $b^{(pr.)}_{t+1}=F^{(1)}(b^{(po.)}_t)$. 
    \al{
    2)~&b^{(po.)}_t(e,b) \notag \\ 
    &= 
    \begin{cases}
    \delta_{(0,p_{11})}(e,b) &\mbox{ if } y(t) \neq \Xi,\\
    \frac{p(b,\Gamma_t(e,b)) b^{(pr.)}_t(e,b)}{\int_{e'} \sum_{b'} p(b',\Gamma_t(e',b')) b^{(pr.)}_t(e',b') \, de'} &\mbox{ if } y(t) = \Xi,
    \end{cases} \label{eq:theta-evolve}
    }
    where $p(b,\Gamma(e,b)):= \bP(y=\Xi \mid b, \Gamma(e,b))$ is the probability that a packet is dropped when the belief state is $b$ and the decision taken is $\Gamma(e,b)$.~We denote this transformation as $b^{(po.)}_{t}=F^{(2)}(b^{(pr.)}_t,\Gamma_t,y(t))$.
\end{lemma}

We now describe the dynamic programming equations, the solution of which yields optimal policies for the estimator.

\begin{theorem} \label{thm:prepost-val-func}
    Consider Problem 2~\eqref{def:new_obj}. Then we can recursively define the value functions $W^{(1)}_t$ and $W^{(2)}_t$ as follows:
    \al{
    & W^{(1)}_0 (b^{(pr.)}) := 0, \label{eq:W0} \\
    & W^{(1)}_{t+1} (b^{(pr.)}) = \inf\limits_{\Gamma \in \mathcal{G}} \bE\lf[\lambda u + W^{(2)}_{t+1}(F^{(2)}(b^{(pr.)},\Gamma,y)\rt], \notag \\
    & W^{(2)}_{t+1} (b^{(po.)}) \notag \\
    & = \inf\limits_{\hat{e} \in \bR} \bE \lf[(e^+ - \hat{e})^2 + \beta W^{(1)}_t(F^{(1)}(b^{(po.)}))\rt], \label{eq:W2}
    }
    where $\mathcal{G}$ is the set of all measurable functions from $\bR \times [0,1]$ to $\{0,1\}$. Then, for each realization of pre-scheduling belief $b^{(pr.)}$, and post-scheduling belief $b^{(po.)}$, we can find an estimator's prescription $\Gamma\ust$, and an estimate $\hat{e}\ust$, respectively, whose performance is arbitrarily close to the optimal performance~\eqref{def:new_obj}~\cite[Prop. 9.19]{bertsekas1996stochastic}.
\end{theorem}

\section{Jointly Optimal Policies} \label{sec:joint_optimality}
We will now prove structural results for an estimator and scheduler that are jointly optimal for Problem 1, cf.~\eqref{def:obj}. We begin by stating some assumptions that are required for the structural results. The following is a stability assumption and ensures existence of finite solutions for the average cost~\eqref{def:avg_obj} and $\beta$-discounted cost ~\eqref{def:obj} problems.

\begin{assumption}\label{assum:stabl}
    The Markovian channel probability $p_{01}$ and the system parameter $a$ satisfy
    \nal{
    a^2(1-p_{01})<1.
    }
\end{assumption}

The following is a common assumption on GE channels~\cite{Yao2022age, dutta2023optimal}. We will use this to analyze the structural properties of an optimal scheduling policy.

\begin{assumption}\label{assum:channel}
The Markovian channel parameters satisfy $p_{11} \geq p_{01}$. 
\end{assumption}


\begin{definition}[Symmetric unimodal about a center] \label{def:SU}
Let $r\in\mathbb{R}$. We define the set of symmetric unimodal functions about a center $r$ as
\nal{
& \mathcal{S}(r)
:=\big\{\nu:\mathbb{R}\to\mathbb{R}\ \big|\ \ 
|x-r|\le |y-r|\ \text{implies } \\
& \hspace{3.5cm} \nu(x)\ge \nu(y), \forall x,y\in\mathbb{R} \big\}.
}
Also, we use the convention that $\delta_r(\cdot) \in \cS(r)$. 
\end{definition}

The following theorem is the main result of this section.
\begin{theorem} \label{thm:opt_est}
    Let Assumptions~\ref{assump:noise}--\ref{assum:channel} hold. Then for the optimization Problem , cf.~\eqref{def:new_obj}, and hence, Problem 1, cf.~\eqref{def:obj}, we have
    \begin{enumerate}
        \item There exists an optimal estimation strategy such that $\hat{e}(t)=0$. More specifically, we have that $\hat{x}(-1) = x_0$ and for $t \geq 0$,
        \al{
        & \hat{x}(t)=
        \begin{cases}
        a\hat{x}(t-1) &\mbox{ if } y(t) = \Xi,\\
        x(t) &\mbox{ otherwise}.\label{eq:opt_est}
        \end{cases}
        }
        \item There is no loss of optimality in restricting the optimization of Problem 1 to the class of scheduling policies of the form~\eqref{eq:u(e,b,y)} to the following form
        \al{
        u(t)=\Tilde{f}_t(e(t),b(t)). \label{eq:opt_sch}
        }
        \item For each $e\in\bR$, there exists a threshold $b\ust(e)$ such that it is optimal to transmit only when the belief at the sensor $b \geq b\ust(e)$. Thus, there exists an optimal transmission strategy corresponding to~\eqref{def:obj} that exhibits a threshold structure. Moreover, an optimal transmission strategy in even in $e$.
    \end{enumerate}
\end{theorem}

\begin{proof}
\begin{enumerate}
    \item We have $e(0) = 0$ for all $b \in [0,1]$. Then, $b^{(pr.)}_0(\cdot,b) \in \cS(0)$ for all $b \in [0,1]$ since $b^{(pr.)}_0(e,b) = \delta_{(0,b(0))}(e,b), b(0) \in [0,1]$. Now, it follows from Corollary~\ref{lemma:opt_pres} that there exists an optimal prescription at $t=0$ that is even and increasing. As a result, from Lemma~\ref{lemma:theta-SU} we have that $b^{(po.)}_0(\cdot,b) \in \cS(0)$ for all $b \in [0,1]$. Finally, from Lemma~\ref{lemma:min_0} we have that $\hat{e}(0) =0$ is an optimal estimator. Next, upon applying Lemma~\ref{lemma:pi-SU} we infer that $b^{(pr.)}_1(\cdot,b) \in \cS(0)$, $b \in [0,1]$. Thus, from Corollary~\ref{lemma:opt_pres} we can conclude that there exists an optimal prescription at $t=1$ that is even and increasing. Hence, $b^{(po.)}_1(\cdot,b) \in \cS(0)$, $b \in [0,1]$ (Lemma~\ref{lemma:theta-SU}). Consequently, $\hat{e}(1) =0$ is an optimal estimator (Lemma~\ref{lemma:min_0}). Proceeding in this recursive manner concludes the proof of 1).
    \item First, we fix the optimal estimator as in~\eqref{eq:opt_est}. The optimization problem~\eqref{def:new_obj} now reduces to one which involves a single decision maker, i.e., the sensor. For this problem, we will show that $(e(t),b(t))$ is a controlled MDP with control $u(t)$. Alternatively, we have to prove that $(e(t),b(t))$ is an information state. This is shown in Steps (i) and (ii) below.
\nal{
&\text{(i)}~\bP(e(t+1),b(t+1) \mid e(t),b(t), \{y(s)\}_{s=0}^{t-1},u(t)) \\
& =\sum_{y \in \bR \cup \Xi} \Bigl(\bP(e(t+1) \mid y(t)=y, e(t)) \\
& \times \bP(b(t+1) \mid b(t), y(t)=y,u(t)) \\
& \times \bP(y(t)=y \mid e(t),b(t),u(t))\Bigr) \\
& = \bP(e(t+1),b(t+1) \mid e(t),b(t),u(t)),
}
where the first equality follows from the law of total probability~\cite{grimmett2020probability} and the last equality follows from~\eqref{def:error_evolve},~\eqref{def:beliefevolv}, and~\eqref{eq:prob_y}.

(ii) Using the law of iterated expectation, we can write the instantaneous cost as
\al{
& \bE[(e^+(t))^2 + \lambda u(t) \mid \{e(s),b(s),u(s)\}_{s=0}^{t}, \{y(s)\}_{s=0}^{t-1}] \notag \\
& =\lambda u(t) + e(t)^2 \bP(y(t)=\Xi \mid e(t),b(t),u(t)) \notag \\
& = \lambda u(t) + (1-b(t)) e(t)^2 \mathds{1}({u(t) = 1}) \\
& + e(t)^2 \mathds{1}({u(t) = 0}), \label{eq:cost}
}
where the first equality follows from~\eqref{eq:e^+_t}, and the last equality follows from~\eqref{eq:prob_y}. Thus, the instantaneous cost is a function of the current state and action. Thus, it follows from~\eqref{eq:cost} that once we fix the estimator as in~\eqref{eq:opt_est} we can equivalently write the instantaneous cost which we denote by $\Tilde{d}$ of Problem 1 (see \eqref{def:obj}) as follows:
\al{
	\Tilde{d}(e,b,u) := 
 \begin{cases}
	    e^2 & \text{if } u=0, \\
        (1-b)e^2 + \lambda & \text{if } u=1.
	\end{cases}
 \label{def:d_tilde}
}

\item The proof is given in Appendix~\ref{thm_proof}. 
\end{enumerate}
\end{proof}
\section{Optimal policies for the Average cost problem} \label{sec:avg_cost}
In this section we show the existence of optimal policies for the following \emph{average cost} problem~\eqref{def:avg_obj} as described below using the vanishing discount approach~\cite{schal1993average},
\al{
	\min_{(f,g)} \limsup_{T \rightarrow \infty} \frac{1}{T}\bE_{(f,g)} \left(\sum_{t=0}^{T} \left((x(t)-\hat{x}(t))^2+\lambda u(t) \right)\right).\label{def:avg_obj}
} 
This involves taking the limit of the optimal policies obtained for the $\beta$-discounted cost problem as the discount factor $\beta \uparrow 1$. The next proposition shows this. For this, we first define $V\ub$ to denote the value function for Problem 1~\eqref{def:obj} with estimator fixed as in~\eqref{eq:opt_est}, i.e., $(e,b) \in \bR \times [0,1]$.
\al{
    V\ub(e,b) = \min_{f}\bE_{f} \left(\sum_{t=0}^{\infty} \beta^{t}\Tilde{d}(e(t),b(t),u(t))\right)\label{def:pomdpobj}.
}
where $\Tilde{d}$ is as in~\eqref{def:d_tilde}. 

The following proof follows from~\cite[Theorem 3.8]{schal1993average} under the General Assumption~\cite[p. 164]{schal1993average}, Assumptions (S) and (B) of~\cite{schal1993average}.

\begin{proposition} \label{prop:vanish_discount}
    Consider Problem 1 and~\eqref{def:avg_obj}, and let Assumptions~\ref{assump:noise}--\ref{assum:channel} hold. Let $\{\beta_n\}_{n \in \bN}$, $\beta_n \in (0,1)$ be a sequence of discount factors such that $\beta_n \uparrow 1$, so that $\{\Tilde{f}^{(\beta_n)}\}_{n \in \bN}$ is the corresponding sequence of optimal policies. Then for any $(e,b) \in \bR \times [0,1]$, there exists a subsequence $\beta_{n_k} \uparrow 1$, $k \in \bN$ of $\{\beta_n\}$ such that $\Tilde{f}\ust(e,b) := \lim_{k \rightarrow \infty} \Tilde{f}^{(\beta_{n_k})}(e,b)$, i.e., $\Tilde{f}\ust$ is the limit point of $\{\Tilde{f}^{(\beta_{n_k})}\}_{k \in \bN}$. Then, $\Tilde{f}\ust$ is an optimal policy for the average cost problem~\eqref{def:avg_obj}. Furthermore, if $\rho\ust$ is the optimal average cost of~\eqref{def:avg_obj} and $V\ub$ is as in~\eqref{def:pomdpobj}, then $\rho\ust = \lim_{\beta \uparrow 1} (1-\beta) V\ub$.
\end{proposition}


We then have the following result for the average cost problem~\eqref{def:avg_obj}.
\begin{theorem} \label{thm:joint_result}
    Let Assumptions~\ref{assump:noise}--\ref{assum:channel} hold. Then for~\eqref{def:avg_obj}:
    \begin{itemize}
        \item[1)]  There exists an optimal estimation strategy $\Tilde{g}\ust$ that satisfies~\eqref{eq:opt_est}.

        \item[2)] There exists an optimal transmission policy $\Tilde{f}\ust$ that has a threshold structure, i.e., for every $e \in \bR$, there exists a threshold $b\ust(e)$ such that the sensor transmits only when $b \geq b\ust(e)$. Moreover, an transmission policy is even in $e$.
    \end{itemize}
\end{theorem}

\begin{proof}
    1) The estimator~\eqref{eq:opt_est} for the $\beta$-discounted cost continues to remain optimal for the average cost. This is because the optimal estimator $\hat{e}$ minimizes only the immediate estimation error and does not affect the future cost term as can been seen from~\eqref{eq:hat_e_indep}.
    
    2) We first consider a sequence of discount factors $(\beta_n)_{n \in \bN}$ such that $\beta_n \uparrow 1$. We then have from Theorem~\ref{thm:opt_est}-3) that for each $\beta_n$, there exists an optimal policy that is of the threshold-type and even in $e$. Since the limit of a sequence of threshold-type policies is also of the threshold-type, the proof follows from Proposition~\ref{prop:vanish_discount}.
\end{proof} 
\section{Numerical results} \label{sec:simulation}
In this section, we fix the optimal estimator $\Tilde{g}$ as in~\eqref{eq:opt_est} and focus on computing an optimal transmission strategy $\Tilde{f}$ for~\eqref{def:avg_obj}. We first use relative value iteration (RVI) to compute an optimal policy for~\eqref{def:avg_obj} in Section~\ref{subsec:RVI_simulation}. Moreover, we use VI to compute an optimal policy for~\eqref{def:obj}. However, RVI/VI requires knowledge of the Markovian channel parameters $p_{01}, p_{11}$, which might not be available in practice. To address this, we propose a (fully data-driven) AC algorithm~\cite{konda2003onactor} that utilizes the structural result of an optimal transmission policy as shown in Theorem~\ref{thm:joint_result}-2) and learns an efficient policy in Section~\ref{subsec:AC}. As discussed below, the structural result of Theorem~\ref{thm:joint_result} reduces the policy search space, and also facilitates an efficient implementation of the AC algorithm. For implementation of RVI/VI and AC we truncate the state space to $\cX:= [-L, L] \times [0,1]$. Throughout, we set the transmission price and other parameters as follows: $\lambda = 0.65$, $a = 0.7$, $p_{01} = 0.4$, $p_{11} = 0.7$, and $L = 2$, respectively. 

\subsection{Relative Value Iteration} \label{subsec:RVI_simulation}

As is shown in~\cite[Theorem 5.5.4]{Hernandez2012discrete}, under Assumptions 4.2.1 and 5.5.1 of~\cite{Hernandez2012discrete}, there exist a constant $\rho\ust$ which is the optimal average cost of~\eqref{def:avg_obj}, and a continuous function $h: \bR \times [0,1] \rightarrow \bR$ that satisfy the average cost optimality equation, i.e., for $(e,b) \in \bR \times [0,1]$
\al{
\rho\ust + h(e,b) = \min_{u \in \{0,1\}} & \Big[\Td(e,b,u) + \sum_{b' \in \cB} \int_{e' \in \bR} h(e',b') \notag \\
& \times \Tilde{p}(e',b' | e,b,u) \,de'\Big], \label{eq:ACOE}
}
where $\Tilde{d}$ is as in~\eqref{def:d_tilde}. Also, there exists a deterministic stationary optimal policy $\Tilde{f}$ that satisfies the right hand side of~\eqref{eq:ACOE}. We will now use RVI algorithm~\ref{alg:rvi}~\cite[p. 101]{Hernandez2012discrete} to solve~\eqref{eq:ACOE}. For implementing RVI, we first discretize the truncated space in which error $e$ resides, $[-L,L]$ with $0.2$ quantization width. The belief $b(t)$ assumes values in the set $\{\cT^k(b(0))\}_{k=1}^{K} \cup \{\cT^k(p_{01})\}_{k=1}^{K} \cup \{\cT^k(p_{11})\}_{k=1}^{K}$, where we let $K = 10$ and $b(0) \in [0,1]$. The optimal $n$-stage cost~\cite{Hernandez2012discrete} is recursively updated as shown in~\eqref{eq:g_n} with $\Tilde{f}_n$ being the minimizer of the right hand side of~\eqref{eq:g_n}. Next, we consider two functions $h_n(\cdot, \cdot)$ and $\rho_n(\cdot,\cdot)$ on $\cX$ with updates as in~\eqref{eq:h_n},~\eqref{eq:rho_n}, respectively. Then, it follows from~\cite[Theorem 5.6.3]{Hernandez2012discrete} that under assumptions 4.2.1, 5.5.1, and 5.6.1 of~\cite{Hernandez2012discrete} for all $(e,b) \in \cX$, $\rho_n(e,b) \rightarrow \rho*$ and $h_n(e,b) \rightarrow h(e,b)$. In addition, $\Tilde{f}_n \rightarrow \Tilde{f}$ as shown in~\cite[Theorem 5.6.7]{Hernandez2012discrete}.
\begin{algorithm}
    \caption{Relative Value Iteration (RVI)}\label{alg:rvi}
    \begin{algorithmic}[0]
        \State Initialize: $g_0(e,b)\gets 0$, $h_0(e,b) \gets 0$ for all $(e,b) \in \mathcal{X}$ 
        \State Input: reference state $(e_0,b_0)$ 
        \State Parameters: $\epsilon$
        \For {$ n = 0, 1,2,\dots$}
            \ForAll{$(e,b) \in \mathcal{X}$}
                \vspace{-0.5cm}
                \State \al{& g_{n+1}(e,b) \gets \min_{u \in \{0,1\}} \Big\{\Tilde{d}(e,b,u)  \notag \\
                & + \int_{(e',b') \in \cX} \Tilde{p}(e', b' \mid e,b,u)   g_{n}(e',b') \,de' \Big\} \label{eq:g_n} \\
                & \Tilde{f}_{n+1}(e,b) \in \argmin_{u \in \{0,1\}} \Big\{\Tilde{d}(e,b,u)  \notag \\ 
                & + \int_{(e',b') \in \cX} \Tilde{p}(e', b' \mid e,b,u) g_{n}(e',b') \,de'\Big\} \notag \\
                & h_{n+1}(e,b) \gets g_{n+1}(e,b) - g_{n+1}(e_0,b_0) \label{eq:h_n} \\
                & \rho_{n+1}(e,b) \gets g_{n+1}(e,b) - g_{n}(e,b) \label{eq:rho_n}
                }
            \EndFor
            \If{$\|h_{n+1}-h_{n}\|_{\infty}  < \epsilon$}
                \State \textbf{break}
            \EndIf
        \EndFor
        \State \textbf{return:} $|g_{n+1}(e,b)-g_{n}(e,b)|$ for any $(e,b)$, $\Tilde{f_n}$
    \end{algorithmic}
\end{algorithm}
We compare the performance of a policy that is optimal for the average cost problem~\eqref{def:avg_obj}, and is obtained using RVI, with the following sub-optimal policies:

1) Sub-optimal policy 1: The sensor always attempts a packet transmission irrespective of the current state $(e,b) \in \cX$.

2) Sub-optimal policy 2: The sensor attempts a packet transmission at regular intervals, which we refer to as the \emph{Period}. 

3) Sub-optimal policy 3: The sensor attempts transmission with a probability $p^{(pol.)}$ at each time. Furthermore, these attempts are independent across times.

4) Sub-optimal policy 4: This policy is derived by assuming that the wireless channel state process is i.i.d. over times. The probability that the channel is in good state, is equal to the stationary probability that the Markovian channel is in good state. We then solve the RVI under this assumption to obtain an optimal policy. 

Fig.~\ref{fig:periodic_iid} compares the performance of the optimal policy with the sub-optimal policies 2 and 3 as the period and $p^{(pol.)}$ are varied. It can been observed from Fig.~\ref{fig:periodic_iid} that as the period decreases or $p^{(pol.)}$ increases, the average reward increases. This implies that more frequent transmission attempts result in an increase in average reward. And hence, we take sub-optimal policy 1 to compare the average reward with the optimal policy obtained using RVI. We also compare the performance of the above sub-optimal policies with optimal policy as the system parameter $a$, and Markovian parameters $p_{11}, p_{01}$ are varied, as shown in Fig.~\ref{fig:Avg_reward_subopts}. We fix the transmission attempt probability $p^{(pol.)}$ of the sub-optimal policy 3 to be equal to the average transmission rate of the optimal policy obtained by RVI, and the period of sub-optimal policy 2 is set equal to $1/p^{(pol.)}$. It can be seen from Fig.~\ref{fig:Avg_reward_subopts}(a) that as the parameter $a$ increases, the average reward decreases. This is because with an increase in $a$, the error increases~\eqref{def:error_evolve}. The difference $(p_{11}-p_{01})$ in Fig.~\ref{fig:Avg_reward_subopts}(b) indicates the Markovian nature of the channel. Thus, as $(p_{11}-p_{01})$ increases, it implies that the Markovian nature of the channel becomes more pronounced. Moreover, keeping $p_{11}$ and decreasing $p_{01}$ indicates that the channel is bad, and hence the average reward decreases.

\begin{figure}[htbp]
	\begin{centering}
		\includegraphics[scale=0.33]{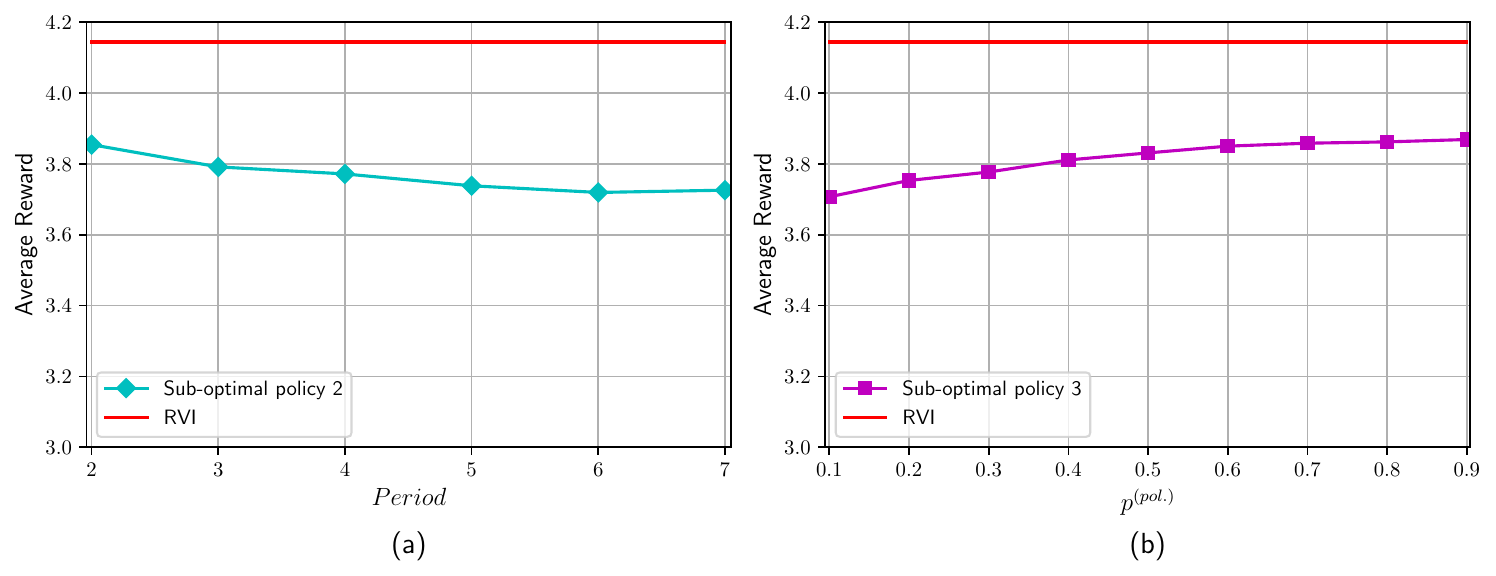} 
		\par\end{centering}
	\caption{Comparison of optimal policy obtained by RVI with $a=0.7, p_{01} = 0.4, p_{11} = 0.7$, (a) sub-optimal policy 2 as the period of transmission is varied; (b) sub-optimal policy 3 as $p^{(pol.)}$ is varied.}
	\label{fig:periodic_iid}
\end{figure}

\begin{figure}[htbp]
	\begin{centering}
		\includegraphics[scale=0.33]{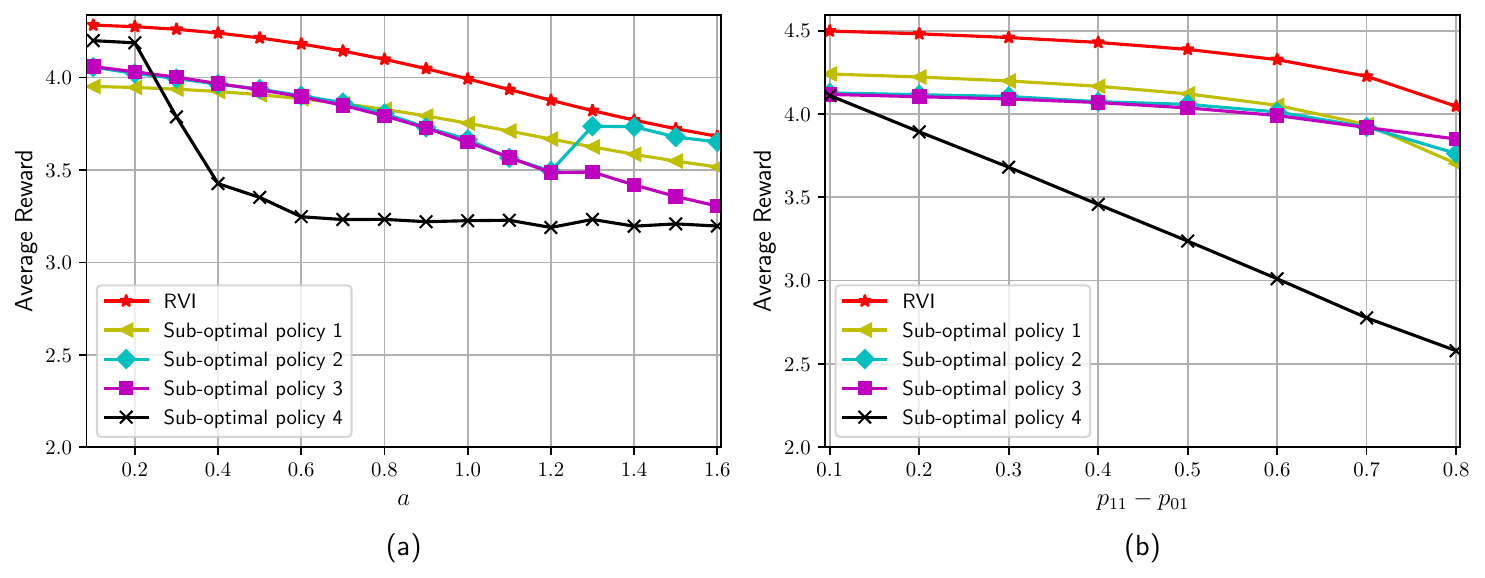} 
		\par\end{centering}
	\caption{Comparison of optimal policy obtained by RVI with the sub-optimal policy as (a) $a$ is varied with $p_{01} = 0.4, p_{11} = 0.7$; (b) $p_{11}-p_{01}$ with $p_{11} = 0.9$.}
	\label{fig:Avg_reward_subopts}
\end{figure}

\subsection{Actor-Critic Algorithm} \label{subsec:AC}
We propose a learning algorithm that can be used to tune the policy when the parameters $a, p_{01}, p_{11}$ are unknown. We leverage the structural result of Theorem~\ref{thm:joint_result} in order to parameterize the transmission policy, with the policy parameters subsequently tuned using the AC algorithm~\cite{konda2003onactor}. The threshold curve $b\ust(\cdot)$ is approximated by a linear curve, i.e., we set $b\ust(e) = \theta_1 + \theta_2 |e|$, where $\theta_1$ and $\theta_2$ denote the intercept and slope of the linear curve respectively. For $\theta = (\theta_1, \theta_2, \theta_3)\in\bR^3$, consider the policy
\nal{
\pi_{\theta}(1|e,b) = \frac{1}{1 + \exp(-\theta_2(b - (\theta_1 + \theta_2 |e|)))}, (e,b) \in \cX.
}
$\pi_{\theta}(1| e,b)$ denotes the probability with which the policy with parameter $\theta$ attempts a packet transmission when the current state is $(e,b)$. Then, the probability of no packet transmission $\pi_{\theta}(0|e,b) = 1 - \pi_{\theta}(1|e,b)$. Let $\rho(\pi_{\theta})$ denote the average cost of $\pi_{\theta}$, i.e.,
\nal{
\rho(\pi_{\theta}) = \limsup_{T \rightarrow \infty} \frac{1}{T} \bE_{(\pi_{\theta}, \Tilde{g})} \lf(\sum_{t=0}^{T-1} \Td(e(t),b(t), u(t))\rt),
}
where $\Td$ is as in \eqref{def:d_tilde}. Define now the advantage function~\cite{konda2003onactor} $\zeta_{\theta} : \cX \times \{0,1\} \rightarrow \bR$ according to 
\nal{
& \zeta_{\theta}(e,b, u) = \Td(e,b,u) - \rho(\pi_{\theta}) \\
& + \bE_{\theta}\lf[h(e(1),b(1)) \mid e(0) = e, b(0) = b, u(0) = u\rt], &\theta \in \bR^3,
}
where $\bE_{\theta}$ denotes that the expectation is w.r.t. the measure induced by the policy $\pi_{\theta}$, and $h(\cdot, \cdot)$ is as in~\eqref{eq:ACOE}. The critic estimates $\zeta_{\theta}$ which is then used by the actor to update its policy by tuning the associated parameter vector $\theta$. This is done by updating $\theta$ in an approximate gradient direction of the cost. We use linear function approximation~\cite{konda2003onactor} for the critic to estimate $\zeta_{\theta}$. Let $m \in \bN$ denote the number of critic parameters, $\omega=(\omega_1, \omega_2, \ldots, \omega_m)^T$ denote the weight vector used by the critic and $\Tilde{\zeta}_{\theta}$ denote the approximation of $\zeta_{\theta}$. Then,
\nal{
\Tilde{\zeta}_{\theta}^{(\omega)}(e,b,u) = \sum_{i=1}^{m} \omega_i \phi_{\theta}^{(i)}(e,b,u), 
}
where $\phi_{\theta} = \lf(\phi_{\theta}^{(1)}, \phi_{\theta}^{(2)}, \ldots, \phi_{\theta}^{(m)}\rt)^T$ denotes the feature vector used by the critic that is dependent on $\theta$. Following~\cite{konda2003onactor}, we adopt the following choice for $\phi_{\theta}$: 
\nal{\phi_{\theta}(e,b,u) = \nabla_{\theta} \log \pi_{\theta}(u \mid e,b),
}
where $\nabla_{\theta}$ denotes the gradient with respect to $\theta$. As a result of the above choice, we have $m =3$. The pseudo-code for the AC algorithm is given in Algorithm~\ref{algo:AC}. We use $\omega(t)$ (cf.~\eqref{eq:critic_update}) and $\theta(t)$ (cf.~\eqref{eq:actor_update}) to denote the critic weight and actor parameter at time $t$, respectively. $z(t)$ (cf.~\eqref{eq:z}) represents the eligibility trace vector at time $t$ \cite{sutton1998reinforcement}. Then, under Assumptions 3.3, 4.1, 4.2, 4.4, 4.5, 4.8, and 4.9 of~\cite{konda2003onactor}, we have the following convergence result for the AC algorithm which is shown in~\cite[Theorem 6.3]{konda2003onactor}.
\begin{theorem}
    Consider the AC algorithm presented in Algorithm~\ref{algo:AC} with TD(1) critic, i.e., TD($\lambda$) with $\lambda=1$ and under function approximation. Then, $\liminf_t |\nabla_{\theta} \rho(\pi_{\theta(t)})| = 0$ with probability 1, where $\rho(\cdot)$ is the average cost function, and $\nabla_{\theta} \rho(\pi_{\theta(t)})$ represents the gradient of $\rho(\pi_{\theta})$ with respect to $\theta$ and evaluated at $\theta(t)$.
\end{theorem}

\begin{algorithm}
	\caption{Actor-Critic Algorithm} \label{algo:AC}
	\begin{algorithmic}[0]
            \State Input: policy parameterization $\pi_{\theta}(u|e,b)$
            \State Input: advantage function parameterization $\tilde{\zeta}^{(\omega)}_{\theta}(e,b,u)$
            \State Parameters: actor step size $\alpha^{(\theta)} >0$, \\
            \hspace{1.7cm} critic step size $\alpha^{(\omega)} > 0$
            \State Initialize $\theta(0)$, $\omega(0)$, $(e(0),b(0)) \in \cX$,  $\hat{\rho}(0)$
            \State Sample $u(0) \sim \pi_{\theta(0)}(\cdot|e(0),b(0))$ 
            \State Evaluate $z(0) = \phi_{\theta(0)}(e(0),b(0),u(0))$
		\For {$t =1,2,\ldots, T$}
			\State Take action $u(t-1)$
                \State Observe cost $\Tilde{d}(e(t-1), b(t-1), u(t-1))$ (cf.~\eqref{def:d_tilde}), and 
                \State next state $(e(t),b(t))$ as described by~\eqref{eq:e^+_t,e_t} and~\eqref{def:beliefevolv}
                \State Sample $u(t) \sim \pi_{\theta(t-1)}(\cdot|e(t),b(t))$
                \State Update average cost: 
                \al{
                & \hat{\rho}(t) = \hat{\rho}(t-1) + \alpha^{(\omega)}\Bigl[\Tilde{d}(e(t-1), b(t-1), u(t-1)) \notag \\
                & \qquad \qquad - \hat{\rho}(t-1) \Bigr] \notag
                }
                \State Update critic: 
                \al{
                & \omega(t) = \omega(t-1) + \alpha^{(\omega)}\bigl[\Tilde{d}(e(t-1), b(t-1), u(t-1)) \notag \\
                -& \hat{\rho}(t-1) + \tilde{\zeta}_{\theta(t-1)}^{(\omega(t-1))}(e(t),b(t),u(t)) \notag \\ 
                -& \tilde{\zeta}_{\theta(t-1)}^{(\omega(t-1))}(e(t-1),b(t-1),u(t-1))\bigr]z(t-1) \label{eq:critic_update}
                }
                \State Update $z$ for TD(1) critic: Let $(\tilde{e}, \tilde{b}) \in \cX$.
                \State \al{z(t) = 
                \begin{cases} \label{eq:z}
                    z(t-1) \\
                    + \phi_{\theta(t-1)}(e(t),b(t),u(t)) &\mbox{if } (e(t),b(t)) \neq (\tilde{e}, \tilde{b}), \\
                    \phi_{\theta(t-1)}(e(t),b(t),u(t)) &\mbox{otherwise} 
                \end{cases}}
                \State Update actor: 
                \al{
                \theta(t) &= \theta(t-1) - \alpha^{(\theta)} \tilde{\zeta}_{\theta(t-1)}^{(\omega(t-1))}(e(t),b(t),u(t)) z(t) \label{eq:actor_update}
                }
		\EndFor
	\end{algorithmic} 
\end{algorithm}
Next, for the purpose of performing simulations, we begin by re-formulating the original cost minimization problem~\eqref{def:avg_obj} as a reward maximization problem. The instantaneous reward function is set to $(L^2 + 1) - \Tilde{d}(e,b,u)$, see \eqref{def:d_tilde}, which is negative of the instantaneous cost function. This is done in order to align the optimization problem with the standard reinforcement learning framework. Fig.~\eqref{fig:polic_curve}(a) compares the performance of AC with RVI and value iteration (VI) for $\beta$-discounted cost~\eqref{def:obj}  \cite{dutta2023optimal} with $\beta = 0.99$. The optimal policy obtained from RVI/VI is expected to outperform AC as, (a) AC only has access to noisy data and not the system model (unlike RVI/VI) and (b) it uses function approximation unlike RVI/VI to ease computational effort that will bring in inaccuracies. The performance of AC is nonetheless close to that of RVI/VI procedures as can been in Fig.~\eqref{fig:polic_curve}(a). Moreover, the performance of VI is same as that of RVI. This result is consistent with Proposition~\ref{prop:vanish_discount}. Fig.~\eqref{fig:polic_curve}(b) depicts the threshold curve $b\ust(e)$ obtained by AC. It can seen that the threshold curve is non-increasing in $|e|$. This implies that as $|e|$ increases the sensor attempts a packet transmission even for a lower value of $b$. Thus it intuitively means that the immediate reduction of error takes precedence over waiting for better channel quality. We will consider proving this structural result in future work. Fig.~\eqref{fig:polic_curve}(c) shows the policy learned by AC. It shows the probability of attempting a packet transmission by the sensor for each value of $(e,b)$. It is even in $e$, and for each value of $e$, the probability of transmitting a packet increases with $b$. Fig.~\eqref{fig:change_in_params} compares the performance of the AC, RVI, and VI as the system parameter $a$ and the channel parameters $p_{01}, p_{11}$ are varied. It can be seen from Fig.~\eqref{fig:change_in_params}(a) that as $|a|$ increases, the average reward decreases. This is because with an increase in $|a|$, the error $e$~\eqref{def:error_evolve} also increases resulting in increase in instantaneous cost $\Tilde{d}$~\eqref{def:d_tilde} and hence, decrease in average reward. Also, it can be seen from Fig.~\eqref{fig:change_in_params}(b) that as the Markovian channel becomes good, i.e., $p_{11}$ increases, the average reward also increases. The optimal policy is seen to outperform the convergent policy of the AC. This is however expected as RVI/VI provably find the optimal policy. The performance of AC suffers from two constraints: (a) AC only has access to noisy data and not the system model (unlike RVI/VI) and (b) it uses function approximation unlike RVI/VI to ease computational effort that in turn brings inaccuracies. The performance of AC is nonetheless close to that of the RVI/VI procedures.

Finally, it can be observed from Fig.~\eqref{fig:change_in_params}(c) that as $p_{11}-p_{01}$ increases, the average reward decreases.
\begin{figure}[htbp]
	\begin{centering}
		\includegraphics[scale=0.34]{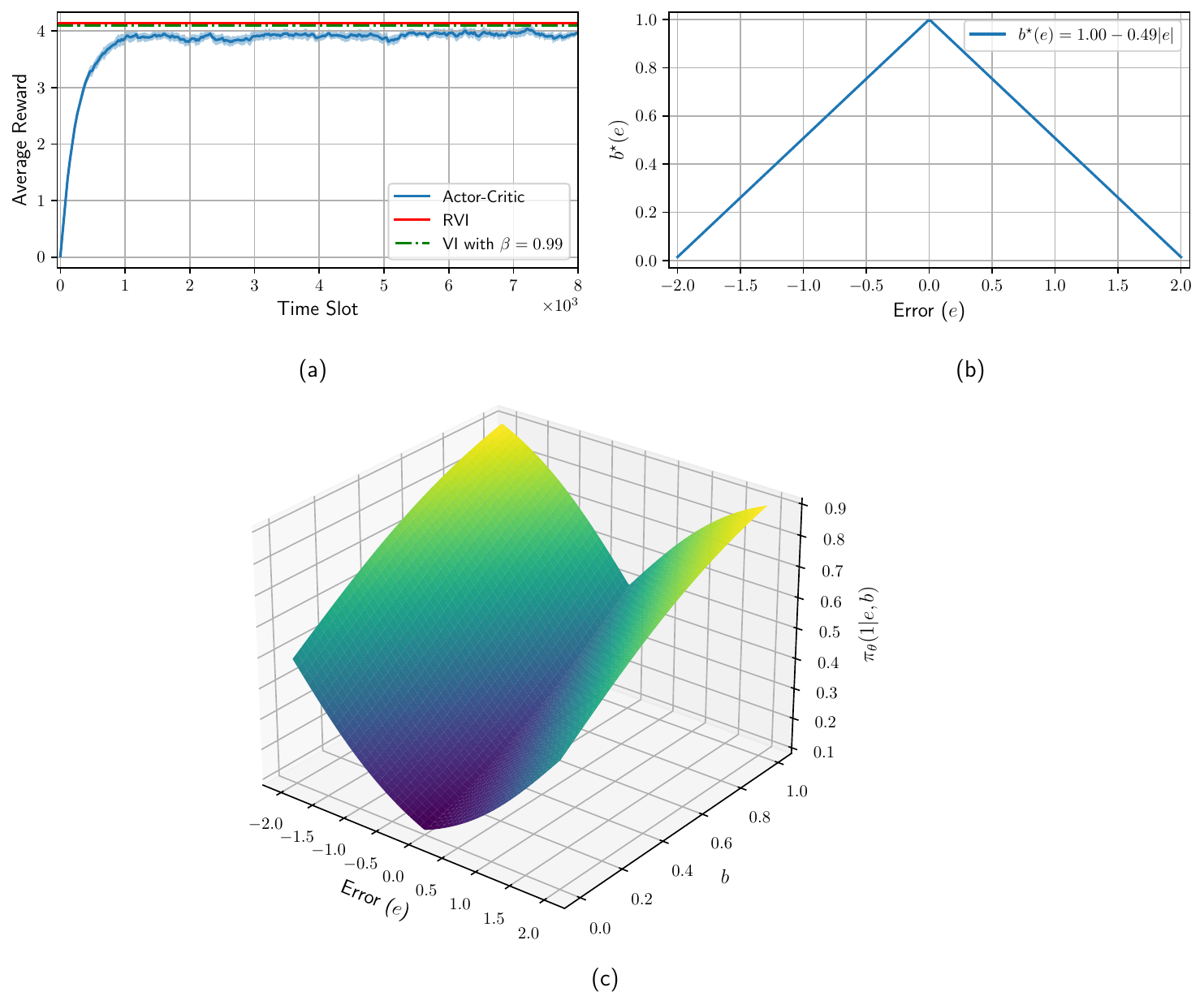} 
		\par\end{centering}
	\caption{(a) Comparison of the average reward achieved by the AC, RVI, and the VI with $T=8000,\beta = 0.99,~\alpha^{(\theta)} = 0.0001$, $\alpha^{(w)} = 0.002,~a = 0.7,~p_{01} = 0.4$, and $p_{11} = 0.7$: (b) Threshold curve, $\tau^{(1)}(q) = 0.99 - 0.99|e|$ obtained using the AC; (c) Policy obtained from the AC.}
	\label{fig:polic_curve}
\end{figure}

\begin{figure}[htbp]
	\begin{centering}
		\includegraphics[scale=0.35]{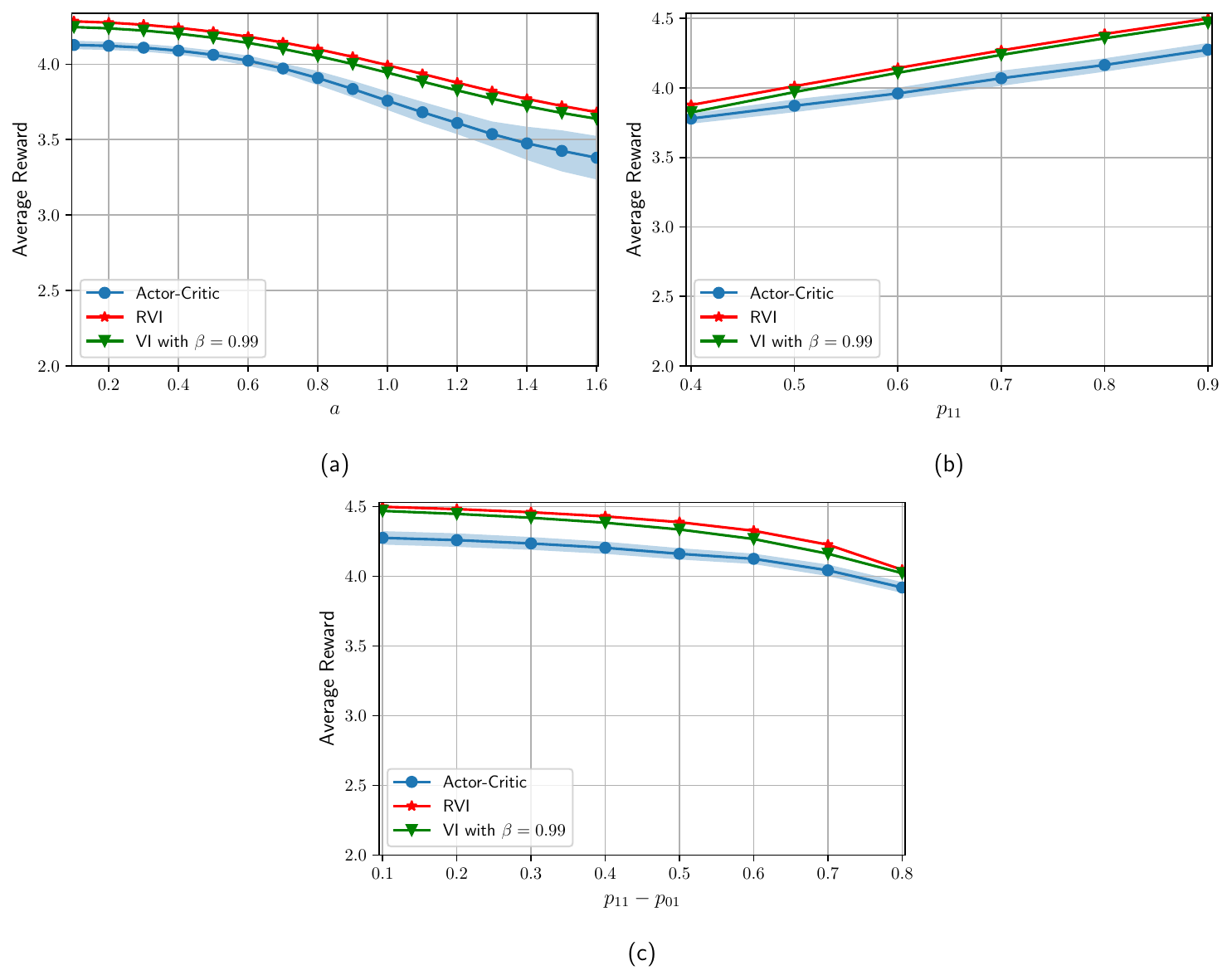} 
		\par\end{centering}
	\caption{Performance comparison of the AC, RVI and VI as system and channel parameters are varied. We use $T=8000$, while step-sizes of the AC are held fixed at $\alpha^{(\theta)} = 0.0001$, and $\alpha^{(w)} = 0.002$: (a) $p_{01} = 0.4, p_{11} = 0.7$; (b) $a = 0.7, p_{01} = p_{11} - 0.1$; (c) $a=0.7$, $p_{11} = 0.9$.}
	\label{fig:change_in_params}
\end{figure}

\section{Conclusion}

We consider a remote estimation setup consisting of a sensor communicating observations to a remote estimator via a GE channel. The goal is to minimize the expected value of an infinite-horizon cumulative discounted cost consisting of squared estimated error and transmission power consumed by the sensor. For this decentralized stochastic control problem we design optimal transmission and estimation strategies. Our analysis relies upon the common information approach which was devised in~\cite{nayyar2013decentralized} to solve this decentralized stochastic control problem. More specifically, we first formulate this problem as a POMDP, and then prove structural results for optimal transmission and estimation strategies. Specifically, we show that there exists a transmission strategy that has a threshold structure with respect to the belief state, and a ``Kalman-like'' estimation strategy, that are jointly optimal. This structural result admits an efficient parameterization of the policies, and helps to restrict the search space of policies to the class of threshold-type policies. This enables us to use the actor-critic (AC) algorithm for optimizing the cost when the system parameters are unknown. The performance of AC is however close to the RVI/VI procedures despite (a) not having access to transition probabilities and (b) using function approximation for the value function that inherently brings in inaccuracies.
\appendices

\section{Preliminary results for Section~\ref{sec:joint_optimality}}
The proof of the following result is given in~\cite[Lemma 11]{Nayyar2013optimal}.
\begin{lemma} \label{lemma:convolve_SU}
    Let $\nu_1 \in \cS(r)$ and $\nu_2 \in \cS(0)$ be two pdfs. Then, $\nu_1 \ast \nu_2 \in \cS(r)$.
\end{lemma}

The following three lemmas describe some properties of the pre-scheduling and post-scheduling beliefs.
\begin{lemma} \label{lemma:min_0}
    Let $b^{(po.)}$ be a post-scheduling belief such that $b^{(po.)}(\cdot,b) \in \cS(0)$ for all $b \in \cB$. Then, 0 achieves the minimum value in~\eqref{eq:W2}.
\end{lemma}

\begin{proof}
    The right-hand side of~\eqref{eq:W2} can be expanded as
    \al{
    W^{(2)}_{t+1}(b^{(po.)}) & = \beta W^{(1)}_t(F^{(1)}(b^{(po.)})) \notag \\
    & + \inf\limits_{\hat{e} \in \bR} \sum_{b \in \cB} \int_{e \in \bR} (e-\hat{e})^2 b^{(po.)}(e,b) \,de. \label{eq:hat_e_indep}
    }
    Since $b^{(po.)}\in\cS(0)$, the above expression implies that,
    \nal{
    0 \in \argmin_{\hat{e} \in \bR} \sum_{b \in \cB} \int_{e \in \bR} (e-\hat{e})^2 b^{(po.)}(e,b) \,de.
    }
    This concludes the proof.
\end{proof}

\begin{definition}[Even and increasing function ]
    We say that a function $\Gamma: \bR \times [0,1] \rightarrow \bR_+$ is even and increasing if,
    \begin{itemize}
        \item[i)] $\Gamma(\cdot,b)$ is even for each $b \in \cB$, i.e., $\Gamma(e,b)=\Gamma(|e|,b)$,
        \item[ii)] $\Gamma(e,b)$ is non-decreasing in $|e|$ for each $b \in \cB$.
    \end{itemize}
\end{definition}

The following two lemmas show the relationship between pre-scheduling and post-scheduling beliefs.

\begin{lemma} \label{lemma:theta-SU}
Suppose that the following hold:
    \begin{itemize}
        \item[i)] $b^{(pr.)}_t(\cdot,b) \in \cS(0)$ for all $b \in \cB$,
        \item[ii)] $\Gamma_t(\cdot,b)$ is even and increasing.
    \end{itemize}
    Then, the post-scheduling beliefs satisfy $\theta_t(\cdot,b) \in \cS(0)$ for all $b$.
\end{lemma}

\begin{proof}
    The proof follows from Lemma~\ref{lemma:prepost_evolve}. Specifically,
    \begin{itemize}
        \item[1)] If $y(t) \neq \Xi$, then $\theta_t(e,b)=\delta_{(0,p_{11})}(e,b)$. Thus, $\theta_t(\cdot,b) \in \cS(0)$, $b \in \cB$.
        \item[2)] If $y(t) = \Xi$, then we have $p(b,\Gamma_t(e,b))=p(b,\Gamma_t(-e,b))$, and $p(b,\cdot)$ are non-increasing for each $b \in \cB$. This implies that $p(b,\cdot)\in \cS(0)$ for each $b \in \cB$. Since the product of two $\cS(0)$ functions is $\cS(0)$, hence $\theta_t(\cdot,b) \in \cS(0)$, $b \in \cB$.
    \end{itemize}
    This completes the proof.
\end{proof}

\begin{lemma} \label{lemma:pi-SU}
    If $\theta_t(\cdot,b) \in \cS(0)$ for all $b \in \cB$, then $b^{(pr.)}_{t+1}(\cdot,b) \in \cS(0)$, $b \in \cB$.
\end{lemma}

\begin{proof}
    The proof follows from Lemma~\ref{lemma:prepost_evolve} since from Lemma~\ref{lemma:convolve_SU} the convolution of two $\cS(0)$ functions is $\cS(0)$.
\end{proof}

As a result of the above lemmas, if we want to show that $b^{(po.)}(\cdot,b) \in \cS(0)$, then it is sufficient to show that the conditions of Lemma~\ref{lemma:theta-SU} hold. We use the majorization theory~\cite{hajek2008paging} to show that these conditions hold.
\begin{definition}[Symmetric rearrangement (SR) of a set]
Let $\mathcal{L}(\mathbb{B})$ denote the Lebesgue measure of a Borel measurable set $\mathbb{B} \subset \bR$. Let $\bB$ be a Borel set such that $\mathcal{L}(\bB) < \infty$. Its symmetric rearrangement, which we denote by $\bB_\sigma$, is the open interval centered around 0 such that $\mathcal{L}(\bB)=\mathcal{L}(\bB_\sigma)$.
\end{definition}

\begin{definition}[SR of a function]
    Let $h:\bR \rightarrow \bR_+$ be an integrable function. Its symmetric decreasing rearrangement is defined as,
    \nal{
    h_\sigma(x) := \int_{0}^{\infty} \mathds{1}_{\{z \in \bR: h(z) > t\}_\sigma} (x) \,dt
    }
\end{definition}

\begin{definition}[Majorization]
    Given two pdfs $\mu$ and $\nu$ on $\bR$, we say $\nu$ \emph{majorizes} $\mu$, and denote it as $\mu \prec \nu$, if
    \nal{
    \int_{|x| \leq t} \mu_\sigma(x)
\,dx \leq  \int_{|x| \leq t} \nu_\sigma(x)
\,dx, \quad \text{for all } t>0.
}

The above condition can equivalently be written as follows. $\mu \prec \nu$ if for any Borel set $\bB \subset \bR, \cL(\bB) < \infty$, there exists another Borel set $\bB' \subset \bR$ with $\cL(\bB)=\cL(\bB')$ and 
\nal{
\int_{\bB} \mu \,dx \leq \int_{\bB'} \nu \,dx.
}
\end{definition}

The next three results follow from~\cite[Appendix A]{Nayyar2013optimal}.
\begin{lemma}\label{lemma:major_convolve}
    Let $\mu \in \cS(0)$ be a pdf. Then, given two pdfs $\nu$ and $\Tilde{\nu}$ such that $\nu \prec \Tilde{\nu}$, and $\Tilde{\nu} \in \cS(r)$ for some $r \in \bR$, we have that $\nu \ast \mu \prec \Tilde{\nu} \ast \mu$.
\end{lemma}

\begin{lemma}\label{lemma:H-L-ineq}
    Let $h$ and $g$ be two non-negative integrable functions on $\bR$. Then,
    \nal{
    \int_{\bR} h(x) g(x) \,dx \leq \int_{\bR} h_{\sigma}(x) g_{\sigma}(x) \,dx.
    }
\end{lemma}

\begin{lemma}\label{lemma:pdf_SU}
    If $\mu$ and $\nu$ are two pdfs on $\bR$ such that $\mu \prec \nu$, then for any symmetric and unimodal function $h$,
    \nal{
    \int_{\bR} \mu_{\sigma}(x) h(x) \,dx \leq \int_{\bR} \nu_{\sigma}(x) h(x) \,dx.
    }
\end{lemma}

We next define a relation and its property for pre-scheduling and post-scheduling beliefs. 

\begin{definition}[Relation $\mathcal{R}$]
    Let $\mu$ and $\nu$ be two pre-scheduling beliefs (post-scheduling beliefs) from $\bR \times [0,1]$ to $\bR_+$. We say $\mu \cR \nu$ if,
    \begin{itemize}
        \item[1)] for each $b \in \cB$, $\mu(\cdot,b) \prec \nu(\cdot,b)$,
        \item[2)] for all $b \in \cB$, $\nu(\cdot,b) \in \cS(e)$ for some point $e \in \bR$.
    \end{itemize}
\end{definition}

\begin{definition}[Property $\cR$]
We say that a value function $W: \bR_+ \rightarrow \bR$
satisfies property $\cR$, if $\mu \cR \nu$ implies that $W(\mu) \geq W(\nu)$ $W^{(1)}(\mu) \geq W^{(1)}(\nu) (W^{(2)}(\mu) \geq W^{(2)}(\nu))$.

\end{definition}
The above property intuitively means that $\nu$ is a better pdf for estimating the value of error $e$~\eqref{def:e} in comparison to $\mu$ since it is more symmetrically concentrated around $e$. Next, we present some results that are used for proving Proposition~\ref{prop:W_R}.
\begin{lemma}\label{lemma:R_F(theta)}
    If $b^{(po.)}_t \cR \Tilde{b}^{(po.)}_t$, $t \in \bZ_+$, then $F^{(1)}(b^{(po.)}_t)$ $\cR$ $F^{(1)}(\Tilde{b}^{(po.)}_t)$, where $F^{(1)}$ is the transformation of $b^{(pr.)}$ given $b^{(po.)}$ as shown in Lemma~\ref{lemma:prepost_evolve}.
\end{lemma}
\begin{proof}
    Let $b^{(pr.)}_{t+1}=F^{(1)}(b^{(po.)}_t)$ and $\Tilde{b}^{(pr.)}_{t+1}=F^{(1)}(\Tilde{b}^{(po.)}_t)$. Then, from Lemma~\ref{lemma:prepost_evolve} we have
    \nal{
    b^{(pr.)}_{t+1}(e,b)= \sum_{b' \in \cB} \bP(b_{t+1}=b \mid b_t =b') \eta_t(e,b') \ast \varphi,
    }
    where $\eta_t(e,b') = (1/|a|) b^{(po.)}_t(e/a,b')$. Similarly
    \nal{
    \Tilde{b}^{(pr.)}_{t+1}(e,b)= \sum_{b' \in \cB} \bP(b_{t+1}=b \mid b_t =b') \Tilde{\eta}_t(e,b') \ast \varphi,
    }
    where $\Tilde{\eta}_t(e,b') = (1/|a|) \Tilde{b}^{(po.)}_t(e/a,b')$.

    Now, in order to show $b^{(pr.)}_{t+1} \cR \Tilde{b}^{(pr.)}_{t+1}$, it suffices to show
    \begin{itemize}
        \item[1)] for each $b' \in \cB$, we have $\varphi \ast \eta_t(\cdot,b') \prec \varphi \ast \Tilde{\eta}_t(\cdot,b')$,
        \item[2)] for all $b'$, we have $\varphi \ast \Tilde{\eta}_t(\cdot,b')\in \cS(e')$ for some $e' \in \bR$.
    \end{itemize}
    
    First we consider 1). Since we are given $b^{(po.)}_t(\cdot,b) \prec \Tilde{b}^{(po.)}_t(\cdot,b)$, it follows that $\eta_t(\cdot,b') \prec \Tilde{\eta}_t(\cdot,b')$. This combined with Lemma~\ref{lemma:major_convolve} proves 1).

    Next we consider 2). Since we are given that for all $b \in \cB, \Tilde{b}^{(po.)}_t(\cdot,b)\in \cS(e)$ for some $e \in \bR$ , it follows that $\Tilde{\eta}_t(\cdot,b') \in \cS(e')$ for some $e' \in \bR$ ($e$ and $e'$ may be different except for the case when $e=e'=0$) and $b' \in \cB$ . This combined with Lemma~\ref{lemma:convolve_SU} proves that $\varphi \ast \Tilde{\eta}_t(\cdot,b') \in \cS(e')$. This shows 2). This concludes the proof.
\end{proof}

\begin{lemma}\label{lemma:R_L}
    Given a post-scheduling belief $b^{(po.)}$, let $L(b^{(po.)}):=\inf\limits_{\hat{e} \in \bR} \sum_{b \in \cB} \int_{e \in \bR} (e-\hat{e})^2 b^{(po.)}(e,b) \,de$. For any $b^{(po.)}$ and $\Tilde{b}^{(po.)}$ such that $b^{(po.)} \cR \Tilde{b}^{(po.)}$, it follows that $L(b^{(po.)}) \geq L(\Tilde{b}^{(po.)})$, i.e., $L$ satisfies property $\cR$.
\end{lemma}

\begin{proof}
    Given $b^{(po.)} \cR \Tilde{b}^{(po.)}$, we assume $\Tilde{b}^{(po.)}(\cdot,b) \in \cS(e')$ for some $e' \in \bR$ and for all $b \in \cB$. Now, for a constant $m>0$, consider a function $h(e,m)=m-\min\{m,(e-\hat{e})^2\}$. Then for any $b \in \cB$,
    \nal{
    \int_{e \in \bR} h(e,m) &b^{(po.)}(e,b) \,de  \leq \int_{e \in \bR} h_{\sigma}(e,m) b^{(po.)}_{\sigma}(e,b) \,de \\
    & =\int_{e \in \bR} (m-\min\{m,e^2\}) b^{(po.)}_{\sigma}(e,b) \,de \\
    & \leq \int_{e \in \bR} (m-\min\{m,e^2\}) \Tilde{b}^{(po.)}_{\sigma}(e,b) \,de \\
    & = \int_{e \in \bR} (m-\min\{m,(e-e')^2\}) \Tilde{b}^{(po.)}(e,b) \,de,
    }
    where the first inequality follows from Lemma~\ref{lemma:H-L-ineq}. The first equality follows from the SR of the function $h(\cdot,\cdot)$. The last inequality follows from Lemma~\ref{lemma:pdf_SU} along with the assumption that $b^{(po.)} \cR \Tilde{b}^{(po.)}$. Finally, the last equality follows since $\Tilde{b}^{(po.)}(\cdot,b) \in \cS(e')$.~Now, the above inequality implies that for any $\hat{x} \in \bR$ we have,
    \nal{
    & \int_{e \in \bR} (m-\min\{m,(e-\hat{e})^2\}) b^{(po.)}(e,b) \,de  \\
    & \leq \int_{e \in \bR} (m-\min\{m,(e-e')^2\}) \Tilde{b}^{(po.)}(e,b) \,de.
    }
    Since $\int_{e \in \bR} b^{(po.)}(e,b) \,de = \int_{e \in \bR} \Tilde{b}^{(po.)}(e,b) \,de$, the above inequality can be further written as follows,
    \nal{
    & \int_{e \in \bR} (\min\{m,(e-\hat{e})^2\}) b^{(po.)}(e,b) \,de  \\
    & \geq \int_{e \in \bR} (\min\{m,(e-e')^2\}) \Tilde{b}^{(po.)}(e,b) \,de.
    }
    Upon letting $m\to\infty$ in the above expression it follows from monotone convergence theorem~\cite{grimmett2020probability} that, 
    \nal{
    & \int_{e \in \bR} (e-\hat{e})^2 b^{(po.)}(e,b) \,de \geq \int_{e \in \bR} (e-e')^2 \Tilde{b}^{(po.)}(e,b) \,de.
    }
    Summing the above for all $b \in \cB$ preserves the inequality, i.e., 
    \nal{
    & \sum_{b \in \cB} \int_{e \in \bR} (e-\hat{e})^2 b^{(po.)}(e,b) \,de  \\
    & \geq \sum_{b \in \cB} \int_{e \in \bR} (e-e')^2 \Tilde{b}^{(po.)}(e,b) \,de.
    }
    Finally, taking infimum over $\hat{e}$ in the left-hand side and then taking infimum over $e'$ in the right-hand side proves the lemma.
\end{proof}

\begin{definition}\label{def:restrict}
    For a pdf $\mu:\bR \rightarrow \bR_+$ and a Borel set $\bB$, we use $\mu_{\bB}$ to denote the restriction of $\mu$ to $\bB$, i.e., 
    \nal{
    \mu_{\bB} (x) = 
    \begin{cases}
        \frac{\mu(x)}{\int_{x' \in \bB} \mu(x') \,dx'} & \mbox{ if } x \in \bB,\\
        0 &\mbox{ otherwise.}
    \end{cases}
    }
\end{definition}
The following two results are essentially Lemma 7 and Lemma 8 of~\cite{Lipsa2011remote}.
\begin{lemma}\label{lemma:symm_int}
    Let $\Tilde{b}^{(pr.)}(\cdot,b) \in \cS(0)$, $b \in \cB$ be a pre-scheduling belief. Then for any $b$ and $0< m(b) \leq 1$, there exists a symmetric interval $\mathcal{J}$ centered around $0$ such that
    \nal{
    \Tilde{b}^{(pr.)}_{\mathcal{J}}(\cdot,b) \succ \Tilde{b}^{(pr.)}_{\mathcal{J}'}(\cdot,b) \text{ and } \int_{\cJ} \Tilde{b}^{(pr.)}(e,b) \,de = 1-m(b),
    }
    for any Borel set $\cJ' \subset \bR$ such that $\int_{\cJ'} \Tilde{b}^{(pr.)}(e,b) \,de = 1-m(b)$.
\end{lemma}

\begin{lemma}\label{lemma:pi_tildepi}
    Let $b^{(pr.)}$ and $\Tilde{b}^{(pr.)}$ be two pre-transmission beliefs such that $b^{(pr.)} \cR \Tilde{b}^{(pr.)}$ with $\Tilde{b}^{(pr.)}(\cdot,b) \in \cS(0)$. Let $\cJ$ and $m(b)$ be as in the above lemma for $\Tilde{b}^{(pr.)}$. Given $b$, if $\cJ'$ is such that $\int_{\cJ'} b^{(pr.)}(e,b) \,de = 1-m(b)$, then $b^{(pr.)}_{\cJ'}(\cdot,b) \prec \Tilde{b}^{(pr.)}_{\cJ}(\cdot,b)$.
\end{lemma}

The following proposition is crucial for proving joint optimality.
\begin{proposition}\label{prop:W_R}
    The value functions $W^{(1)}_t$ and $W^{(2)}_t$ of Theorem~\ref{thm:prepost-val-func} satisfy property $\cR$.
\end{proposition}

\begin{proof}
    We will show the result using induction. Since $W^{(1)}_0(b^{(pr.)})=0$~\eqref{eq:W0} for any pre-scheduling belief $b^{(pr.)}$, hence $W^{(1)}_t$ satisfies property $\cR$ for $t=0$. Next, assume that $W^{(1)}_s$ satisfy property $\cR$ for $s=0,1,\ldots,t$. We now divide the proof into two steps. In this first step we show that $W^{(2)}_{t+1}$ satisfies the property $\cR$. In the second step we show that $W^{(1)}_{t+1}$ also satisfies the property $\cR$. 

    \textbf{Step I:} If $W^{(1)}_t$ satisfies the property $\cR$, then $W^{(2)}_{t+1}$ also satisfies $\cR$: As in Lemma~\ref{lemma:min_0}, we first expand $W^{(2)}_{t+1}$ as follows,
    \nal{
    W^{(2)}_{t+1}(b^{(po.)}) = &~ \beta W^{(1)}_t(F^{(1)}(b^{(po.)})) \notag \\
    & + \inf\limits_{\hat{e} \in \bR} \sum_{b \in \cB} \int_{e' \in \bR} (e'-\hat{e})^2 b^{(po.)}(e',b) \,de'.
    }
    We will show that each term on the right-hand side of the above equation satisfies property $\cR$. Now, Lemma~\ref{lemma:R_F(theta)} combined with the assumption that $W_t^1$ satisfies property $\cR$ implies that the first term on right-hand side satisfies property $\cR$. Lemma~\ref{lemma:R_L} shows that the second term also satisfies property $\cR$. This completes Step I.

    \textbf{Step II:} If $W^{(2)}_{t+1}$ satisfies $\cR$, then $W^{(1)}_{t+1}$ also satisfies $\cR$: We consider $b^{(pr.)}$ and $\Tilde{b}^{(pr.)}$ such that $b^{(pr.)} \cR \Tilde{b}^{(pr.)}$ and $\Tilde{b}^{(pr.)}(\cdot,b) \in \cS(e')$ for some $e' \in \bR$ and for all $b \in \cB$. From Theorem~\ref{thm:prepost-val-func} we have,
    \nal{
    W^{(1)}_{t+1} (b^{(pr.)}) & = \inf\limits_{\Hat{\Gamma} \in \mathcal{G}} \bE\lf[\lambda u + W^{(2)}_{t+1}(F^{(2)}(b^{(pr.)},\Gamma,y)\rt] \\
    & = \inf\limits_{\Hat{\Gamma}:\bR \times [0,1] \rightarrow \{0,1\}} \widetilde{W}^{(1)}_{t+1}(b^{(pr.)},\hat{\Gamma}),
    }
    where $\widetilde{W}^{(1)}_{t+1}(b^{(pr.)},\hat{\Gamma}) := \bE\lf[\lambda u + W^{(2)}_{t+1}(F^{(2)}(b^{(pr.)},\Gamma,y)\rt]$. We will show that for any given prescription $\Gamma$, we can construct another prescription $\Tilde{\Gamma}$ such that $\widetilde{W}^{(1)}_{t+1}(b^{(pr.)},\Gamma) \geq \widetilde{W}^{(1)}_{t+1}(\Tilde{b}^{(pr.)},\Tilde{\Gamma})$. This in turn will imply that $W^{(1)}_{t+1} (b^{(pr.)}) \geq W^{(1)}_{t+1} (\Tilde{b}^{(pr.)})$ which will complete Step II. We first expand $\widetilde{W}^{(1)}_{t+1}$ as follows,
    \nal{
    & \widetilde{W}^{(1)}_{t+1}(b^{(pr.)},\Gamma) = \bE\lf[\lambda u + W^{(2)}_{t+1}(F^{(2)}(b^{(pr.)},\Gamma,y)\rt] \\
    & = \lambda \sum_{b \in \cB} \int_{e \in \bR} \Gamma(e,b) b^{(pr.)}(e,b) \,de \\
    & + \bE[W^{(2)}_{t+1}(F^{(2)}(b^{(pr.)},\Gamma,y)] \\
    & = \lambda \sum_b \int_e \Gamma(e,b) b^{(pr.)}(e,b) \,de + \sum_b \int_e b^{(pr.)}(e,b) \\
    & \times (1-p(b,\Gamma(e,b))) W^{(2)}_{t+1}(\delta_{(0,p_{11})}(e,b)) \,de \\
    & + \sum_b \int_e b^{(pr.)}(e,b) p(b,\Gamma(e,b)) W^{(2)}_{t+1}(F^{(2)}(b^{(pr.)},\Gamma,y)) \,de,
    }
    where the last equality follows from~\eqref{eq:theta-evolve}. We now note that $W^{(2)}_{t+1}(\delta_{(0,p_{11})}(e,b))$ does not depend on $(e,b)$, and hence assume $L:=W^{(2)}_{t+1}(\delta_{(0,p_{11})}(e,b))$ for some constant $L \in \bR$. The above expression for $\widetilde{W}^{(1)}_{t+1}(b^{(pr.)},\Gamma)$ can now be written as,
    \al{
    & \widetilde{W}^{(1)}_{t+1}(b^{(pr.)},\Gamma) = \lambda \sum_b \int_e \Gamma(e,b) b^{(pr.)}(e,b) \,de \notag \\
    & + L \sum_b \int_e b^{(pr.)}(e,b) (1-p(b,\Gamma(e,b)))  \,de \notag \\
    & + \sum_b \int_e b^{(pr.)}(e,b) p(b,\Gamma(e,b)) W^{(2)}_{t+1}(F^{(2)}(b^{(pr.)},\Gamma,y)) \,de.\label{eq:W_tilde}
    }
    For any given prescription $\Gamma$, we will now construct another prescription $\Tilde{\Gamma}$ as described next. For this purpose, we define $e\ust(b, b^{(pr.)})>0$ to be the radius of an open interval centered at $e'$ (recall, $\Tilde{b}^{(pr.)}(\cdot,b) \in \cS(e')$) such that 
    \al{
    & \int_{|e-e'| < e\ust(b,b^{(pr.)})} \Tilde{b}^{(pr.)}(e,b) \,de \notag \\
    & = \int_e b^{(pr.)}(e,b) (1-\Gamma(e,b)) \,de.\label{eq:define_r} 
    }
    Next, we define $\Tilde{\Gamma}$ as follows,
    \al{
    \Tilde{\Gamma}(e,b) = 
    \begin{cases}
        0 & \mbox{ if } |e-e'| < e\ust(b,b^{(pr.)}),\\
        1 & \mbox{ if } |e-e'| \geq e\ust(b,b^{(pr.)}).
    \end{cases} \label{eq:gamma_tilde}
    }
    Now, to show that the above choice of $\Tilde{\Gamma}$ satisfies $\widetilde{W}^{(1)}_{t+1}(b^{(pr.)},\Gamma) \geq \widetilde{W}^{(1)}_{t+1}(\Tilde{b}^{(pr.)},\Tilde{\Gamma})$, we begin by proving that the following hold for the three terms on the right-hand side of~\eqref{eq:W_tilde},
    \nal{
    \text{A)}~ & \int_e b^{(pr.)}(e,b) \Gamma(e,b) \,de = \int_e \Tilde{b}^{(pr.)}(e,b) \Tilde{\Gamma}(e,b) \,de, \\
    \text{B)}~ & \int_e b^{(pr.)}(e,b) p(b,\Gamma(e,b)) \,de \\
    & = \int_e \Tilde{b}^{(pr.)}(e,b) p(b,\Tilde{\Gamma}(e,b)) \,de, \\
    \text{C)}~ & \int_e b^{(pr.)}(e,b) (1-p(b,\Gamma(e,b))) \,de \\
    & = \int_e \Tilde{b}^{(pr.)}(e,b) (1-p(b,\Tilde{\Gamma}(e,b))) \,de.
    }
    We first consider A). It follows from~\eqref{eq:define_r} and~\eqref{eq:gamma_tilde} that
    \nal{
    & \int_e b^{(pr.)}(e,b) \Gamma(e,b) \,de \\
    & = \int_e b^{(pr.)}(e,b) \,de - \int_{|e-e'| < e\ust(b,b^{(pr.)})} \Tilde{b}^{(pr.)}(e,b) \,de \\
    & = \int_e \Tilde{b}^{(pr.)}(e,b) \,de - \int_{|e-e'| < e\ust(b,b^{(pr.)})} \Tilde{b}^{(pr.)}(e,b) \,de \\
    & = \int_{|e-e'| \geq e\ust(b,b^{(pr.)})} \Tilde{b}^{(pr.)}(e,b) \,de \\
    & = \int_{|e-e'| < e\ust(b,b^{(pr.)})} \Tilde{b}^{(pr.)}(e,b) \Tilde{\Gamma}(e,b)\,de \\
    & + \int_{|e-e'| \geq e\ust(b,b^{(pr.)})} \Tilde{b}^{(pr.)}(e,b) \Tilde{\Gamma}(e,b)\,de \\
    & = \int_e \Tilde{b}^{(pr.)}(e,b) \Tilde{\Gamma}(e,b) \,de.
    }
    Next consider B). It follows from i) that
    \nal{
    & \int_e b^{(pr.)}(e,b) p(b,\Gamma(e,b)) \,de \\
    & = \lf[\int_e b^{(pr.)}(e,b) \mathds{1}_{\{\Gamma(e,b)=0\}}  p(b,0) \,de\rt] \\
    & + \lf[\int_e b^{(pr.)}(e,b) \mathds{1}_{\{\Gamma(e,b)=1\}}  p(b,1) \,de\rt]\\
    & = \lf[\int_e \Tilde{b}^{(pr.)}(e,b) \mathds{1}_{\{\Tilde{\Gamma}(e,b)=0\}}  p(b,0) \,de\rt] \\
    & + \lf[\int_e \Tilde{b}^{(pr.)}(e,b) \mathds{1}_{\{\Tilde{\Gamma}(e,b)=1\}}  p(b,1) \,de\rt]\\
    & \int_e \Tilde{b}^{(pr.)}(e,b) p(b,\Tilde{\Gamma}(e,b)) \,de. 
    }
    Similarly, we can show C) also holds. 

    Next, we will show that $F^{(2)}(b^{(pr.)},\Gamma,y) \cR F^{(2)}(\Tilde{b}^{(pr.)},\Tilde{\Gamma},y)$. This will imply $W^{(2)}_{t+1}(b^{(pr.)}, \Gamma) \geq W^{(2)}_{t+1}(\Tilde{b}^{(pr.)},\Tilde{\Gamma})$ since by assumption $W^{(2)}_{t+1}$ satisfies $\cR$. Finally, $W^{(2)}_{t+1}(b^{(pr.)}, \Gamma) \geq W^{(2)}_{t+1}(\Tilde{b}^{(pr.)},\Tilde{\Gamma})$ combined with the above results, A), B), and C) will then establish $\widetilde{W}_{t+1}(b^{(pr.)}, \Gamma) \geq \widetilde{W}_{t+1}(\Tilde{b}^{(pr.)}, \Tilde{\Gamma})$. Thus, Step II will be proved.

    For the subsequent proof, we will consider $\Tilde{b}^{(pr.)}(\cdot,b) \in \cS(0)$, i.e., $e'=0$. Now to show $F^{(2)}(b^{(pr.)},\Gamma,y) \cR F^{(2)}(\Tilde{b}^{(pr.)},\Tilde{\Gamma},y)$, we have to prove that the following holds,
    \begin{itemize}
        \item[i)] for each $b \in \cB$, $F^{(2)}(b^{(pr.)},\Gamma,y)(\cdot,b) \prec F^{(2)}(\Tilde{b}^{(pr.)},\Tilde{\Gamma},y)(\cdot,b)$,
        \item[ii)] for all $b \in \cB$, $F^{(2)}(\Tilde{b}^{(pr.)},\Tilde{\Gamma},y)(\cdot,b) \in \cS(e)$ for some $e \in \bR$. 
    \end{itemize}
    First, consider i). Fix $b \in \cB$. For $y \neq \Xi$, i) holds trivially.~Now we consider the case when $y= \Xi$. Let $\cJ(b) =\{e \in \bR : \Tilde{\Gamma}(e,b) = 0\}$ and $\cJ'(b)= \{e \in \bR : \Gamma(e,b) = 0\}$. Note that $\cJ(b)$ is a symmetric interval centered around 0, while $\cJ'(b)$ may not be an interval. It then follows from A) that, $\int_{\cJ(b)} \Tilde{b}^{(pr.)}(e,b) \, de = \int_{\cJ'(b)} b^{(pr.)}(e,b) \,de$. To see this, we can write the former equality as $\int_{e \in \bR} \Tilde{b}^{(pr.)}(e,b) \mathds{1}_{\{\Tilde{\Gamma}(e,b)=0\}} \, de = \int_{e \in \bR} b^{(pr.)}(e,b) \mathds{1}_{\{\Gamma(e,b)=0\}} \,de$. Then, it follows from Lemma~\ref{lemma:pi_tildepi} that
    \al{\label{eq:pi_rstrct_major}
    b^{(pr.)}_{\cJ'}(\cdot,b) \prec \Tilde{b}^{(pr.)}_{\cJ}(\cdot,b),
    }
    where from Definition~\ref{def:restrict} we have
    \al{
    b^{(pr.)}_{\cJ'}(e,b) = \frac{b^{(pr.)}(e,b)\mathds{1}_{\{\Gamma(e,b)=0\}}}{ \int_{e'} b^{(pr.)}(e',b) \mathds{1}_{\{\Gamma(e',b)=0\}} \,de'}, \label{eq:pi_rstrct}\\
    \Tilde{b}^{(pr.)}_{\cJ}(e,b) = \frac{\Tilde{b}^{(pr.)}(e,b)\mathds{1}_{\{\Tilde{\Gamma}(e,b)=0\}}}{ \int_{e'} \Tilde{b}^{(pr.)}(e',b) \mathds{1}_{\{\Tilde{\Gamma}(e',b)=0\}} \,de'} \notag.
    }
    Next, for ease of notation denote $D(\Gamma,b^{(pr.)}):= \int_{e} \sum_{b} p(b,\Gamma_t(e,b)) b^{(pr.)}_t(e,b) \, de$. Then, it follows from~\eqref{eq:prob_y} and~\eqref{eq:theta-evolve} that
    \nal{
    & b^{(po.)}(e,b) = \frac{ b^{(pr.)}(e,b) \mathds{1}_{\{\Gamma(e,b)=0\}}}{D(\Gamma,b^{(pr.)})} \\
    & + \frac{(1-b) b^{(pr.)}(e,b) \mathds{1}_{\{\Gamma(e,b)=1\}}}{D(\Gamma,b^{(pr.)})} \\
    & = \frac{(1-(1-b)) b^{(pr.)}(e,b) \mathds{1}_{\{\Gamma(e,b)=0\}}}{D(\Gamma,b^{(pr.)})} \\
    & + \frac{(1-b) b^{(pr.)}(e,b) \mathds{1}_{\{\Gamma(e,b)=0\} \cup \{\Gamma(e,b)=1\}}}{D(\Gamma,b^{(pr.)})}  \\
    & = \frac{b b^{(pr.)}_{\cJ'}(e,b) \int_{e'} b^{(pr.)}(e',b) \mathds{1}_{\{\Gamma(e',b)=0\}} \,de'}{D(\Gamma,b^{(pr.)})} \\
    & + \frac{(1-b) b^{(pr.)}(e,b) }{D(\Gamma,b^{(pr.)})},
    }
    where the last equality follows from~\eqref{eq:pi_rstrct}.~Similarly,
    \nal{
    \Tilde{b}^{(po.)}(e,b) = & \frac{b \Tilde{b}^{(pr.)}_{\cJ}(e,b) \int_{e'} \Tilde{b}^{(pr.)}(e',b) \mathds{1}_{\{\Tilde{\Gamma}(e',b)=0\}} \,de'}{D(\Tilde{\Gamma},\Tilde{b}^{(pr.)})} \\
    & + \frac{(1-b) \Tilde{b}^{(pr.)}(e,b) }{D(\Tilde{\Gamma},\Tilde{b}^{(pr.)})}.
    }
    Finally, from~\eqref{eq:pi_rstrct_major} and our assumption made at the beginning of Step II that $b^{(pr.)} \prec \Tilde{b}^{(pr.)}$ combined with the above expressions for $b^{(po.)}$ and $\Tilde{b}^{(po.)}$, we have that $F^{(2)}(b^{(pr.)},\Gamma,y)(\cdot,b) \prec F^{(2)}(\Tilde{b}^{(pr.)},\Tilde{\Gamma},y)(\cdot,b)$. This completes the proof for i).

    Next, consider ii). Since $\Tilde{b}^{(pr.)}(\cdot,b) \in \cS(0)$, and moreover $\Tilde{\Gamma}(\cdot,b)$ is even and increasing, it follows from Lemma~\ref{lemma:theta-SU} that $F^{(2)}(\Tilde{b}^{(pr.)},\Tilde{\Gamma},y)(\cdot,b) \in \cS(0)$ for all $b \in \cB$. This proves ii). This completes the proof of Step II.
\end{proof}

As a result of the above proposition, we have the following result.

\begin{corollary}\label{lemma:opt_pres}
    Let $b^{(pr.)}(\cdot,b) \in \cS(0)$ for all $b \in \cB$. Then, there exists an optimal prescription $\Tg$ that is even and increasing.
\end{corollary}

\begin{proof}
    We first note that for all $b \in \cB, b^{(pr.)}(\cdot,b) \prec b^{(pr.)}(\cdot,b)$ trivially. Suppose the optimal prescription corresponding to $W^{(2)}_t(b^{(pr.)})$ is $\Gamma$. Then, using this $\Gamma$ we can construct another prescription $\Tg$ similar to the one constructed in Step II of Proposition~\ref{prop:W_R} (using $b^{(pr.)}$ instead of $\Tpi$). As shown in Step II, $\Tg$ is also an optimal prescription that is even and increasing. This completes the proof. 
\end{proof}

\section{Proof of Theorem~\ref{thm:opt_est} 3)} \label{thm_proof}
We provide only an outline of the proof since the proof is similar to \cite[Theorem 4.1]{dutta2023optimal}, the key difference between these two is that now the noise process $\{w(t)\}_{t \in \bZ_+}$ satisfies Assumption \ref{assump:noise}, while~\cite{dutta2023optimal} made a more restrictive assumption that $\{w(t)\}_{t \in \bZ_+}$ is Gaussian.~We first note from part 1) of this theorem that once the estimator has been fixed to be as in~\eqref{eq:opt_est}, Problem 2~\eqref{def:obj} becomes that of exclusively designing an optimal transmission policy for the sensor. Let $J\ub(e,b;\phi)$ be the $\beta$-discounted cost~\eqref{def:pomdpobj} incurred by $f$ when the system starts in state $(e,b)\in \bR \times [0,1]$ with estimator fixed as in~\eqref{eq:opt_est},
    \nal{J\ub(e,b;f) := \bE_{f} \left(\sum_{t=0}^{\infty} \beta^{t}\Tilde{d}(e(t),b(t),u(t))\right). 
    }

The following result allows us to use the value iteration (VI)~\cite{Hernandez2012discrete} to compute $V\ub$~\eqref{def:pomdpobj}. 
\begin{lemma} \label{lemma:assump}
Consider the POMDP~\eqref{def:pomdpobj}, and let Assumption~\ref{assum:stabl} hold. The following properties hold:
    \renewcommand{\labelenumi}{P\arabic{enumi})}
    \begin{enumerate}
    \item The one-stage cost function $\Tilde{d}(e,b,u)$ \eqref{def:d_tilde} is continuous, non-negative, and inf-compact on $(\bR \times [0,1] \times \{0,1\})$.
    \item The transition kernels $\{P(\cdot,u,\cdot)\}_{u\in \{0,1\}}$ that describe the transition probabilities which result when control $u$ is applied, are strongly continuous.
    \item There exists a policy $f$ for which $J\ub(e,b;f) < \infty$ for each $(e,b) \in \bR \times [0,1]$.
    \end{enumerate}
\end{lemma}
\begin{proof}
    The proof follows from~\cite[Lemma 3.1]{dutta2023optimal}.
\end{proof}
We use $\Tilde{p}(e_+,b_+|e,b,u)$ to denote the transition density function for POMDP~\eqref{def:pomdpobj} from current state $(e,b)$ to next state $(e_+,b_+)$ when action $u$ is applied. Next, we divide the proof of Theorem~\ref{thm:opt_est} 3) into the following steps:

\textbf{Step i):} Under Assumption~\ref{assum:stabl}, we can show that we can use value iteration (VI) with VI functions $\{V_n\}_{n \in \bZ_+}$. This follows from Lemma~\ref{lemma:assump} and~\cite[Proposition 3.1]{dutta2023optimal}.   

\textbf{Step ii):} We show that $V\ub(e, b)$ and any optimal transmission strategy are even in $e$ for each $b \in \cB$. This follows from~\cite[Proposition 3.2]{dutta2023optimal}.

\textbf{Step iii):} We construct a simple folded POMDP with the following transition density function similar to~\cite{dutta2023optimal} to ease the analysis
\nal{
p_{fold}(\Te_+, \Tb_+ | \Te, \Tb, \Tu) & = \tilde{p}(\Te_+, \Tb_+ | \Te, \Tb, \Tu) \\
& + \tilde{p}(-\Te_+, \Tb_+ | \Te, \Tb, \Tu),
}
where $\Te_+, \Te \in \bR_+$, $\Tb_+, \Tb \in \cB$. We use $\TV$ to denote the $\beta$-discounted value function for the folded POMDP. $\TV$ is analogous to~\eqref{def:pomdpobj}. We next show that the folded POMDP is equivalent to original POMDP~\eqref{def:pomdpobj}, the proof of which follows from~\cite[Proposition 3.3]{dutta2023optimal}. Hence, it suffices to use the folded POMDP for further analysis.

\textbf{Step iv):} We show that $\TV$ satisfies the following properties:
\begin{enumerate}
    \item[A)] $\TV(\Te, \Tb)$ is non-decreasing with respect to $\Te \in \bR_+$ for each $\Tb \in \cB$.

    \item[B)] $\TV(\Te, \Tb)$ is non-increasing with respect to $\Tb \in \cB$ for each $\Te \in \bR_+$.

    \item[C)] For beliefs $\Tb, \Tb_1, \Tb_2, \Tb_3$ such that $\Tb_1 \ge \Tb_2$ and $\Tb_3 = \Tb \Tb_1 + (1- \Tb) \Tb_2$ we have that
    \nal{\label{eq:c}
            &(1-\Tb)\lambda + \Tb  \Tilde{V}\ub(\Te,\Tb_1) \notag  \\
            &  + (1-\Tb) \Tilde{V}\ub(\Te,\Tb_2)  
             \geq \Tilde{V}\ub(\Te,\Tb_3) .
        }
        
    \item[D)] For each $\Te\in\bR_+$, there exists a threshold $\Tb\ust(\Te)$ such that it is optimal to transmit for the sensor only when $\Tb \geq \Tb\ust(\Te)$. Thus, the optimal transmission strategy corresponding to $\Tilde{V}\ub$ exhibits a threshold structure. 
\end{enumerate}
The proof for the above properties follows from~\cite[Theorem 4.1]{dutta2023optimal} by replacing $\psi(v)$ and $\Tilde{\psi}(v,s)$ in~\cite{dutta2023optimal} by $\varphi(v)$ and $\varphi(v-s) + \varphi(v+s)$, respectively.

\textbf{Step iv):} Finally, it follows from Steps iii) and iv) that there exists an optimal transmission strategy corresponding~\eqref{def:obj} that exhibits a threshold structure.

\section*{References}
\vspace{-0.7cm}
\bibliographystyle{IEEEtran}
\bibliography{refs}

\end{document}